\documentclass[12pt]{amsart}
\usepackage[usenames]{color}

\usepackage{soul}
\usepackage[normalem]{ulem}

\usepackage{a4,amsmath,amssymb,amsthm,amsfonts,enumerate,url,mathbbol,mathrsfs,tikz,graphicx,pifont,}
\usepackage{hyperref}
\usepackage{pgfplots}
\usepgfplotslibrary{polar}

\usetikzlibrary{automata,positioning}
	
\usepackage[utf8]{inputenc}

\usepackage{verbatim}
\usepackage{array}

\def\mG{\mathsf{G}}	

\def\mV{\mathsf{V}}

\def\mE{\mathsf{E}}

\def\mv{\mathsf{v}}
\def\mw{\mathsf{w}}
\def\mz{\mathsf{z}}
\def\mf{\mathsf{f}}
\def\me{\mathsf{e}}

\DeclareMathOperator{\diag}{diag}

\DeclareMathOperator{\Real}{Re}
\DeclareMathOperator{\Ima}{Im}
\DeclareMathOperator{\Id}{Id}

\newtheorem{lemma}{Lemma}[section]
\newtheorem{cor}[lemma]{Corollary}
\newtheorem{theorem}[lemma]{Theorem}
\newtheorem{rem}[lemma]{Remark}
\newtheorem{prop}[lemma]{Proposition}
\newtheorem{ex}[lemma]{Example}
\newtheorem{defi}[lemma]{Definition}
\newtheorem{ass}[lemma]{Assumption}

\def\:{\thinspace:\thinspace}
\def\linie{\vrule height 14pt depth 5pt}

\newcommand\sign{\mathop{\rm sign}}

\newcommand{\bound}{\mathcal{L}}

\newcommand{\f}{\mathfrak{a}}
\newcommand{\h}{\widetilde{H}}
\newcommand{\R}{\mathbb{R}}
\newcommand{\C}{\mathbb{C}}

\newcommand{\N}{\mathbb{N}}
\newcommand{\Z}{\mathbb{Z}}

\newcommand{\cmv}{{c_\mV}}

\newcommand{\dist}{\emph{dist\,}}
\numberwithin{equation}{section}

\title{Bi-Laplacians on graphs and networks}



\author{Federica Gregorio}
\address{Federica Gregorio, Lehrgebiet Analysis, Fakult\"at Mathematik und Informatik, Fern\-Universit\"at in Hagen, D-58084 Hagen, Germany \textrm{ and  }Dipartimento di Ingegneria dell’Informazione, Ingegneria Elettrica e Matematica Applicata, Università degli Studi di Salerno, Via Giovanni Paolo II, 132, 84084 Fisciano (Sa), Italy}
\email{fgregorio@unisa.it}

\author{Delio Mugnolo}
\address{Delio Mugnolo, Lehrgebiet Analysis, Fakult\"at Mathematik und Informatik, Fern\-Universit\"at in Hagen, D-58084 Hagen, Germany}
\email{delio.mugnolo@fernuni-hagen.de}

\thanks{
The first author is member of the Gruppo Nazionale per l'Analisi Matematica,
la Probabilit\`a e le loro Applicazioni (GNAMPA) of the Istituto Nazionale di Alta Matematica
(INdAM). The second author has been partially supported by the Center for Interdisciplinary Research (ZiF) in Bielefeld, Germany, within the framework of the cooperation group on ``Discrete and continuous models in the theory of networks''; and by the Deutsche Forschungsgemeinschaft (Grant 397230547). \\
The authors are grateful to Jochen Glück (Ulm) for countless illuminating discussions.}

\begin{document}

\begin{abstract}
We study the differential operator $A=\frac{d^4}{dx^4}$ acting on a connected network $\mathcal{G}$ along with $\mathcal L^2$, the square of the discrete Laplacian acting on a connected discrete graph $\mG$. For both operators we discuss well-posedness of the associated {linear} parabolic problems
\[
 \frac{\partial u}{\partial t}=-Au,\qquad\frac{df}{dt}=-\mathcal L^2 f,
\]
on $L^p(\mathcal{G})$ or $\ell^p(\mV)$, respectively, for $1\leq p\leq\infty$. In view of the well-known lack of parabolic maximum principle for all elliptic differential operators of order $2N$ for $N>1$, our most surprising finding is that, after some transient time, the parabolic equations driven by $-A$ may display Markovian features, depending on the imposed transmission conditions in the vertices. Analogous results seem to be unknown in the case of general domains and even bounded intervals. Our analysis is based on a detailed study of bi-harmonic functions complemented by simple combinatorial arguments. We elaborate on analogous issues for the discrete bi-Laplacian; a characterization of complete graphs in terms of the Markovian property of the semigroup generated by  $-\mathcal L^2$ is also presented.
\end{abstract}

\maketitle

\section{Introduction}

Our main aim in this article is to discuss qualitative properties of one-dimensional parabolic evolution equations associated with the bi-Laplacian, focusing on the settings of graphs and networks. We are especially going to investigate how different transmission conditions in the vertices of a network can arouse rather different behaviors of solutions to such partial differential equations.

Ever since the pioneering work of Bernoulli and Euler, linear partial differential equations that are second-order in time and fourth-order in space have been fundamental tools in the modeling of large flexible bodies. The application of these models makes it particularly natural to study such linear \textit{beam equations}
\[
\frac{\partial^2 u}{\partial t^2}(t,x)=-\Delta^2 u(t,x),\qquad t\ge 0,\ x\in \Omega\ ,
\]
on \textit{bounded} domains $\Omega$.
Just like in the case of the classical (second-order in space) wave equation, no universally appropriate set of boundary conditions for the beam equation exist: instead, boundary conditions have to be adapted to the underlying physical model. 
In the simplest 1-dimensional case of a (thin, linear) beam $\Omega=(0,1)$, the most usual conditions are
\begin{itemize}
\item (clamped) $u(t,i)=u_x(t,i)=0$;
\item (free) $u_{xx}(t,i)=u_{xxx}(t,i)=0$;
\item (hinged) $u(t,i)=u_{xx}(t,i)=0$;
\item (sliding) $u_x(t,i)=u_{xxx}(t,i)=0$;
\end{itemize}
at both ends $i\in \{0,1\}$. Their physical meaning is described in detail, e.g., in the survey~\cite[\S~2]{HanBenWei99}.

In the Hilbert space setting, bi-Laplacians with hinged and sliding conditions, in particular, are the square of the common Laplacians $\Delta_D, \Delta_N$ with Dirichlet or Neumann boundary conditions, respectively, and therefore both their spectrum and their parabolic properties are easily studied: indeed, both $i\Delta_D,i\Delta_N$ are group generators by Stone's theorem, hence (minus) their squares $-\Delta^2_D,-\Delta^2_N$ generate cosine operator functions
\[
\cos(-t\Delta^2_D)=\frac{e^{it\Delta_D}+e^{-it\Delta_D}}{2},\quad \cos(-t\Delta^2_N)=\frac{e^{it\Delta_N}+e^{-it\Delta_N}}{2},
\]
(\cite[Exa.~3.14.15]{AreBatHie01}) and the corresponding second-order (in time) evolution equation is well-posed.
Furthermore, it is well-known that cosine operator function generators -- whether self-adjoint or not -- enjoy good operator theoretic properties; in particular, they also generate analytic semigroups
\[
e^{-t\Delta^2_D}=\int_0^\infty \frac{e^{-\frac{s^2}{4t}}}{\sqrt{\pi t}}\cos(-t\Delta^2_D) ds,\quad 
e^{-t\Delta^2_N}=\int_0^\infty \frac{e^{-\frac{s^2}{4t}}}{\sqrt{\pi t}}\cos(-t\Delta^2_N) ds,
\]
\cite[Cor.~3.7.15]{AreBatHie01}. This shows well-posedness  of the parabolic equation
\begin{equation}\label{eq:parab4}
u_{t}(t,x)=-\Delta^2 u(t,x),\qquad t\ge 0,\ x\in \Omega,
\end{equation}
with either hinged or sliding boundary conditions.

Such fourth-order parabolic equations display remarkable features that have perhaps not yet received their fair share of attention in the evolution equation community: apart from arising as singular limits of telegraph equations as both the propagation speed and the damping tend to infinity (an observation that goes back to Hadamard, see~\cite{Eng92} for a comprehensive discussion), parabolic equations associated with lower order perturbations of fourth-order operators have been discussed in relation with simplified versions of thin film, Cahn--Hilliard, and Kuramoto--Sivashinsky equations, e.g.\ in~\cite{LauPug00,KohOtt02,GiaKnuOtt08} or in the recent monograph~\cite{GalMitPoz14}; whereas models featuring fourth order noise terms appear in high energy physics~\cite{MueShoWon05}. It is well known that the non-constant solutions of~\eqref{eq:parab4} cannot be positive on $\R_+\times \Omega$, hence in particular they cannot be interpreted as probability distributions; it is therefore all the more intriguing that several stochastic interpretations of~\eqref{eq:parab4} have been proposed so far: a connection with the expected value of the solution of the Schrödinger equation $\frac{\partial u}{\partial t} = i\Delta u$ has been suggested in~\cite{GriHer69}, whereas it has been argued in~\cite{Kry60,Hoc78,Fun79} that~\eqref{eq:parab4} can be understood in terms of pseudo-processes that generalize standard Brownian motions; a further stochastic interpretation that elucidates connections with random walks in the complex plane has been pointed out in~\cite{BonMaz15}.

To the best of our knowledge, the semigroup generated by minus the  bi-Laplacian 
has been first systematically studied by Davies, with a particular focus on the lower dimensional case in~\cite[\S~6]{Dav95} and in~\cite{Dav95b}. Let us summarize his findings as follows.

\begin{theorem}[Davies 1995]\label{thm:davies}
Consider the operator 
\[
H:=\frac{d^4}{dx^4}
\]
on $L^2(\R)$ with maximal domain $D(H)=H^4(\R)$. Then $-H$ is self-adjoint and positive semidefinite on $L^2(\R)$, hence it generates an analytic semigroup of angle $\frac{\pi}{2}$ on $L^2(\R)$. This semigroup maps $L^2(\R)$ to $L^\infty(\R)$ and its integral kernel $p(t,\cdot,\cdot)$ is given by
\[
p(t,x,y)=\frac{1}{2\pi}\int_\R e^{i(x-y)\xi-\xi^4 t}\ d\xi;
\]
$p(t,\cdot,\cdot)$ lies in the Schwartz space ${\mathcal S}(\R\times\R)$  for all $t\in \C_+$. Thus, $(e^{-tH})_{t\ge 0}$ extends to a consistent family of strongly continuous semigroups on $L^p(\R)$ that are analytic of angle $\frac{\pi}{2}$ for all $p\in [1,\infty)$; the spectra of their generators are independent of $p$.
This semigroup satisfies
\[
\|e^{-tH}\|_2=1\quad \hbox{and}\quad \|e^{-tH}\|_1=\|e^{-sH}\|_1>1\quad\hbox{ for all }t,s>0
\]
as well as
\[
\|e^{-tH}\|_{2\to\infty}\le Ct^{-\frac{1}{8}}e^{\frac{t}{2}},\qquad t>0 .
\]
\end{theorem}

In Theorem~\ref{thm:davies} as well as in the whole paper $\Vert T\Vert_{p\to q}$ stands for the operator norm of $T:L^p\to L^q$.

The integral kernel $p(t,x):=p(t,x,x)$ was studied by Hochberg, where on-diagonal estimates where studied; in particular, it was observed in~\cite[\S~2]{Hoc78} that $p(t,x)=p(1,t^{-\frac{1}{4}}x)t^{-\frac{1}{4}}$ with
\[p(1,x)\approx kx^{-\frac{1}{3}}\exp\left({-\frac{3}{8}\left(\frac{x^4}{4}\right)^\frac{1}{3}}\right)\cos\left(\frac{3\sqrt{3}}{8}\left(\frac{x^4}{4}\right)^\frac{1}{3} \right);
\]
for large $x$; here $k$ is a positive constant and the approximation holds up to lower order terms. This shows in particular an oscillatory character that prevents the integral kernel $p$ from being positive; indeed, $p$ changes sign infinitely often.

Davies then went on extending the theory to fourth order differential operators in divergence form with non-constant coefficients and proving that some of these properties hold also in higher dimensions, while some other fail.
Our aim in this paper is to consider a different extension of the 1-dimensional theory: we will glue together finitely many intervals in a graph-like fashion to form objects usually called \textit{network} or \textit{metric graph} in the literature, from~\cite{Lum80,KotSmi97} up to the recent compendia~\cite{BerKuc13,Mug14}. In doing so, we lose the possibility of finding a semigroup of convolution operators; the semigroup generated by minus the bi-Laplacian still consists of integral operators, but to the best of our knowledge its integral kernel is not explicitly known.

The properties of linear PDEs driven by the bi-Laplacian on networks of thin linear beams have been studied often in the literature, with a special focus on controllability and stabilization issues, see e.g.~\cite{LagLeuSch94,DekNic99}. Let us also mention that a theory of elasticity on discrete structures that can be described as graphs $\mG$ is well established and based on the linear dynamical system
\[
\frac{d^2f}{dt^2}(t,\mv)=-\mathcal Lf(t,\mv),\qquad t\ge 0,\ \mv\in \mV,
\]
where $\mathcal L$ is the discrete Laplacian on a graph $\mG$ with vertex set $\mV$. In analogy with the refinements of elasticity models that historically lead to study (space-continuous) beam equations,  it is then natural to study the properties of the operator associated with discretized beam equation: indeed, the discrete \textit{bi}-Laplacian $\mathcal L^2$ was discussed already in~\cite[\S~I.5]{CouFriLew28}. Unlike $\mathcal L$, however, its properties have not been intensively studied afterwards: we will discuss them along with (differential) bi-Laplacians on networks. In the case of second-order parabolic equations, heat semigroups on both graphs and networks are known to be associated with Dirichlet forms; likewise, we are going to show that also semigroups generated by bi-Laplacians on graphs and networks share many relevant features.

As already mentioned, there is no canonical choice of boundary conditions for the bi-Laplacian on domains; the same can be said about transmission conditions in the vertices of a networks. An early physical derivation of different vertex conditions on serially connected beams  has been obtained in~\cite{CheDelKra87}; later studies on this subject include~\cite{DekNic99,DekNic00,KiiKurUsm15}. 
In this paper we follow a somewhat reverse approach, studying infinitely many realizations of the differential operators, especially those that satisfy continuity in the vertices and hence, arguably, better mirror the connectivity of the graph; we then single out transmission conditions for the parabolic equation \textit{depending on} qualitative properties of the parabolic equation: we will be particularly interested in those conditions that induce behaviors analogous to those in Theorem~\ref{thm:davies} and, more generally, those that resemble the case of usual (second-order) diffusion equation. 

What qualitative properties are we interested in, since we wish to focus on parabolic equations? We have already mentioned that the integral kernel of the semigroup generated by the fourth derivative on $\mathbb R$ is known to be non-positive hence the semigroup is a non-positivity preserving semigroup: non-positive semigroup in short. More generally, parabolic maximum principles fail to hold for all diffusion-type problems associated with fourth-order differential equations, in sharp contrast to the second-order case. There follows the necessity to develop an alternative method -- based on ultracontractive estimates and developed already in~\cite{Dav95} -- in order to establish an $L^p$-theory. Lack of positivity also makes studying long-time behavior of solutions of (1.1) more complicated, since the classical (infinite-dimensional) Perron-Frobenius theory is no more available. It turns out that an effective proxy for positivity is \emph{eventual} positivity: may solutions to a Cauchy problem become positive for \textit{large enough} time provided the initial condition is positive?

Relying on the explicit formula of the kernel on $\R$, in \cite{GazGru08} the authors proved that for continuous, compactly supported, positive initial data $u_0$, the solution to the Cauchy problem \begin{equation*}\label{acp}
\begin{cases} \frac{\partial u}{\partial t}(t,x)=-\Delta^2u(t,x), &\qquad
 t\in  [0,\infty),\ x\in \R, \\ u(0,x)=u_0(x), &\qquad x\in\ \R.\end{cases}\end{equation*} 
 satisfies the following:
For any interval $I\subset\R$ there exists  $T_I=T_I(u_0)>0$ such that 
$u(t,x)>0$ for all  $t\geq T_I$ and $x\in I$; and there exists $\tau=\tau(u_0)>0$ such that for any $t>\tau$ there exists a $x_t\in\R$ such that $u(t,x_t)<0$. A generalization in \cite{FerGazGru08} covers the case of initial data $u_0$ that merely decay suitably fast as $|x|\to\infty$.
{ This implies local eventual positivity of solutions to~\eqref{eq:parab4}; such behavior was commonly believed to be generally true but had not been rigorously observed until an abstract setting for studying eventual positivity of solutions to parabolic problems was recently proposed by Daners, Glück and Kennedy. We will borrow their theory presented in \cite{DanGluKen16,DanGluKen16b} and apply it to the setting of compact networks.
 Additionally, we also study the issue of whether a semigroup satisfies an eventual sub-Markovian property, i.e., eventual $L^\infty$-contractivity on the top of eventual positivity. This seems to be new in the literature and of independent interest;  we stress that physical systems that become (sub-)Markovian after a transient time have often been observed, see e.g.~\cite{SuaSilOpp92,GnuHaa96,HufLinGal} and references therein.

We single out two instances that seem to have been overlooked in previous investigations on the bi-Laplacian: these instances correspond to the Friedrichs and Krein--von Neumann realizations of classical extension theory and are shown to have particularly good and particularly bad parabolic properties, respectively. Between these two extreme cases, we study Markovian properties of the semigroup generated by different realizations of the bi-Laplacian, depending on the transmission conditions and/or the network topology.

The paper is structured as follows. Section~\ref{sec:discrete} provides a short overview of the properties of the bi-Laplacian on discrete graphs: by elementary analytic and combinatorial methods we deduce several results that will serve as helpful analogies whenever turning to differential operators; in particular, a characterization of positivity of the semigroup $(e^{-t\mathcal L^2})_{t\ge 0}$ depending on the structure of the graph is given (this is in sharp contrast with positivity of discrete heat semigroup $(e^{-t\mathcal L})_{t\ge 0}$, which has been observed already in~\cite{Kat54,BeuDen59} to hold for all finite graphs). 
In Section~\ref{sec:gener} we introduce the network setting and  we fully characterize the self-adjoint realizations of the fourth derivative. In Section~\ref{sec:selfa} we obtain generation results in $L^2(\mathcal{G})$: several realizations that have appeared in the literature are shown to be special cases of our general theory. 
In Section~\ref{sec:contra} we focus on contractivity properties of the semigroup generated by $-\frac{d^4}{dx^4}$ in different $L^p$-spaces, $1\leq p\leq\infty$: this is mostly based on ideas borrowed from~\cite{Dav95} in the case of fourth-order parabolic equations on $\R$. In Section~\ref{sec:event-pos} we develop a fourth-order counterpart of classical theory of qualitative properties of heat semigroups associated with Dirichlet forms; we focus on the notion of \textit{eventually positive/sub-Markovian} semigroups, which seems the most appropriate one in this context, for several self-adjoint realizations and especially for the Friedrichs and Krein--von Neumann realization. An auxiliary lemma on eventual Markovian property of a semigroup might be of independent interest: it is presented in the Appendix, Section~\ref{sec:app}, together with the basic definitions and some criteria for eventual positivity of semigroups.

\section{The bi-Laplacian on discrete graphs}\label{sec:discrete}

As a warm-up, let us first focus on the properties of the discrete bi-Laplacian. 
For the convenience of the reader we first recall in a naive way some elementary graph theoretical definitions and refer the reader to~\cite[Chapter 1]{Die05} or~\cite[Appendix A]{Mug14} for more precise definitions.

A (discrete) graph $\mG$ is a couple $(\mV,\mE)$ of finite or countably infinite sets of vertices $\mV$ and edges $\mE$ connecting two vertices: if $\me\in\mE$ connects $\mv,\mw\in\mV$,  we say that $\mv,\mw$ are \textit{incident} with $\me$, or that $\mv,\mw$ are \textit{endpoints} of $\me$, and we write $\mv\stackrel{\me}{\sim}\mw$ (or $\mv\sim\mw$, if no confusion is possible).  A \textit{loop} is an edge that connects a vertex to itself, while $\mG$ has \textit{multiple edges} if there exist two vertices that are simultaneously incident with the same two (or more) edges.

A finite graph is called \textit{connected} whenever, given any two vertices $\mv,\mw\in\mV$, there is a path that connects $\mv$ and $\mw$. A \textit{path} is a connected subgraph consisting of a sequence $\me_1,\ldots,\me_n$ of pairwise distinct edges such that for all $j=2,\ldots,n$ there exists a vertex that is simultaneously an endpoint of both $\me_{j-1},\me_j$. A \textit{cycle} is a connected path in which, additionally, each vertex is an endpoint of exactly two edges. A \textit{tree} is a graph that does not contain cycles.  A graph is called \textit{complete} if every pair of distinct vertices is connected by precisely one edge.

Throughout this section we consider a finite (i.e., $V=|\mV|<\infty$ and $E=|\mE|<\infty$) connected graph $\mG$ that has neither loops nor multiple edges.

 We arbitrarily fix an orientation of $\mG$ and introduce the $V\times E$ \textit{incidence matrix} $\mathcal I$ defined by
\[
{\iota}_{\mv \me}:=\left\{
\begin{array}{ll}
-1 & \hbox{if } \mv \hbox{ is initial endpoint of } \me, \\
+1 & \hbox{if } \mv \hbox{ is terminal endpoint of } \me, \\
0 & \hbox{otherwise}.
\end{array}\right.
\]
 One can then consider the  symmetric sesquilinear form $\mathfrak h$ defined by
\begin{equation}\label{eq:formah}
{\mathfrak h}(f,g):=(\mathcal I^T f,\mathcal I^T g)_{\ell^2(\mV)}\ 
\end{equation}
where $\ell^2(\mV)$ is the $V$-dimensional Euclidean space $\R^V$: the associated operator  $\mathcal L$, defined by
\[
(\mathcal Lf,g)_{\ell^2(\mV)}\stackrel{!}{=}{\mathfrak h}(f,g),
\]
is the so-called \textit{discrete Laplacian}  of the graph $\mG$, we refer to~\cite[Chapt.~2]{Mug14} and references therein for its basic properties. Of course, $\mathcal L$ and its square $\mathcal L^2$ are nothing but matrices; however,  we prefer to formulate the following result in the language of operator theory.

\begin{prop}\label{prop:discr-wellp}
The bounded operator $\mathcal L^2$ is self-adjoint and positive semidefinite on $\ell^2(\mV)$. Accordingly, $-\mathcal L^2$ generates a cosine operator function $(\cos(-t\mathcal L^2))_{t\in \R}$ and an analytic, contractive semigroup $(e^{-t\mathcal L^2})_{t\ge 0}$ of angle $\frac{\pi}{2}$ on $\ell^2(\mV)$.
\end{prop}

Because $\mathcal I$ and hence also $\mathcal L$ and $\mathcal L^2$ have real entries, both  $(\cos(-t\mathcal L^2))_{t\in \R}$ and $(e^{-t\mathcal L^2})_{t\ge 0}$ obviously map real-valued functions to real-valued functions. 

Let us denote by 
\[
\deg(\mv):=\sum\limits_{\me\in\mE}|\iota_{\mv\me}|
\]
the degree of a vertex $\mv$: for any $\mv,\mw,\mz\in \mV$ a tedious but straightforward computation yields
\[
{\mathcal L}_{\mv\mz}{\mathcal L}_{\mz\mw}=\begin{cases}
\deg^2(\mv) & \hbox{if }\mv=\mw=\mz,\\
-\deg(\mv)&\hbox{if }\mv=\mz, \ \mz\sim \mw,\\
-\deg(\mw)&\hbox{if }\mw=\mz, \ \mz\sim \mv,\\
1 &\hbox{if }\mv\sim \mz \sim \mw,\\
0 &\hbox{if }\mv\not\sim \mz \hbox{ or } \mz\not\sim \mw;
\end{cases}
\]
in particular, $\mathcal L^2$ does not depend on the chosen orientation of $\mG$. Then a simple computation shows that the $\mv\mw$-entry of $\mathcal L^2$ is
\begin{equation}\label{eq:bilapldiscr}
{\mathcal L^2}_{\mv\mw}=\begin{cases}
\deg^2(\mv)+\deg(\mv) & \hbox{if }\mv=\mw,\\
|N_\mv\cap N_\mw|-\deg(\mv)-\deg(\mw) &\hbox{if }\mv\sim \mw,\\
|N_\mv\cap N_\mw| &\hbox{if }\mv\ne \mw \hbox{ and }\mv\not\sim \mw,
\end{cases}
\end{equation}
where $N_\mv$  is the \textit{neighborhood} of $\mv$, i.e., the set of all  vertices connected with $\mv$ by one edge; clearly, $|N_\mv|=\deg(\mv)$.

\begin{rem}\label{rem:compl}
1) The complete graph on $V$ vertices, whose bi-Laplacian is  given by~\eqref{eq:bilapldiscr} by
\[
{\mathcal L^2}_{\mv\mw}=\begin{cases}
V(V-1) & \hbox{if }\mv=\mw,\\
-V &\hbox{if }\mv\ne \mw,\\
\end{cases}
\]
generates a Markovian semigroup: indeed, this formula shows that $\mathcal L^2=V\mathcal L$, hence $e^{-z\mathcal L^2}=e^{-Vz\mathcal L}$ for all $z\in \C$.

\begin{figure}[htb]
\begin{tikzpicture}[scale=0.8]
\coordinate (g) at (0,0);
\coordinate (h) at (2,0);
\coordinate (i) at (0,2);
\coordinate (l) at (2,2);

\draw[fill] (g) circle (2pt);
\draw[fill] (h) circle (2pt);
\draw[fill] (i) circle (2pt);
\draw[fill] (l) circle (2pt);

\draw (g) -- (h) -- (l) -- (i) -- (g);
\draw (g) -- (l);
\draw (h) -- (i);
\end{tikzpicture}
\caption{The complete graph on four vertices.}\label{fig:compl}
\end{figure}
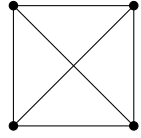

2) Interpreting the parabolic equation associated with the discrete bi-Laplacian as a  diffusion equation with higher diffusive coefficient is generically not true, though, as the following example shows: the path graph on 3 vertices, whose bi-Laplacian is
\[
\mathcal L^2:=
\begin{pmatrix}
2 & -3 & 1\\ -3 & 6 & -3 \\ 1 & -3 & 2
\end{pmatrix}
\quad\hbox{so that}\quad 
e^{-t\mathcal L^2}=
\frac{1}{6}
\begin{pmatrix}
2+3e^{-t}+e^{-9t} & 2-2e^{-9t} & 2-3e^{-t}+e^{-9t}\\
2-2e^{-9t} & 2+4e^{-9t} & 2-2e^{-9t} \\ 
2-3e^{-t}+e^{-9t} & 2-2e^{-9t} & 2+3e^{-t}+e^{-9t}
\end{pmatrix}
\]
shows that $(e^{-t\mathcal L^2})_{t\ge 0}$ needs not be either positive  or $\ell^\infty$-contractive, since e.g.\ for $t=0.1$ one finds the approximate values
\[
\begin{pmatrix}
0.8535 & 0.1978 & -0.0513\\ 0.1978 & 0.6048 & 0.1978\\ -0.0513 & 0.1978 & 0.8535
\end{pmatrix}
\]
showing that the vectors $(0,0,1)^T$ and  $(1,1,0)^T$ are not mapped into the positive cone and into the unit $\infty$-ball, respectively.

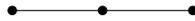
\begin{figure}[htb]
\begin{tikzpicture}[scale=0.8]
\coordinate (g) at (7.5,0);
\coordinate (h) at (9,0);
\coordinate (i) at (10.5,0);

\draw[fill] (g) circle (2pt);
\draw[fill] (h) circle (2pt);
\draw[fill] (i) circle (2pt);

\draw (g) -- (h);
\draw (h) -- (i);
\end{tikzpicture}
\caption{The path graph on three vertices.}\label{fig:path}
\end{figure}

\end{rem}

  In particular, $(e^{-t\mathcal L^2})_{t\ge 0}$ is generally not sub-Markovian; it turns out that complete graphs are actually the only ones on which $-\mathcal L^2$ generates a Markovian semigroup. More precisely, the following holds.

\begin{prop}\label{prop-complete}
The following assertions are equivalent:
\begin{enumerate}[(i)]
\item $\mG$ is complete;
\item $(e^{-t\mathcal L^2})_{t\ge 0}$ is positive;
\item $(e^{-t\mathcal L^2})_{t\ge 0}$ is $\ell^\infty$-contractive.
\end{enumerate}
\end{prop}

Observe that if $\mG$ is complete (and hence connected), then $(e^{-t\mathcal L^2})_{t\ge 0}=(e^{-Vt\mathcal L})_{t\ge 0}$ is also irreducible, since so is $(e^{-t\mathcal L})_{t\ge 0}$ on any connected graph.

\begin{proof}
We have just seen in Remark~\ref{rem:compl} that $(e^{-t\mathcal L^2})_{t\ge 0}$ is Markovian if $\mG$ is complete.

Conversely, let $\mG$ be non-complete, and in particular assume $\mG$ to have at least 3 vertices: we are going to show that the semigroup is neither positive nor $\ell^\infty$-contractive. Take two vertices $\mv,\mw\in \mV$ chosen in such a way that $\mv$ has maximal degree (necessarily $2\le \deg(\mv)\le V-2$), that $\mw$ is not adjacent to $\mv$, and that $\mv,\mw$ have (at least) one common neighbor.

Then, we first observe that by~\eqref{eq:bilapldiscr} the off-diagonal entry ${\mathcal L^2}_{\mv\mw}$ is strictly positive, which prevents $(e^{-t\mathcal L^2})_{t\ge 0}$ from being a positive semigroup.

Furthermore, let us show that the condition
\[-{\mathcal L^2}_{\mv\mv}+\sum_{\mz\ne \mv}|{\mathcal L^2}_{\mv\mz}|\le 0
\]
fails to hold for this choice of $\mv$: by~\cite[Lemma~6.1]{Mug07} this will imply that  $(e^{-t\mathcal L^2})_{t\ge 0}$ is not $\ell^\infty$-contractive. Indeed,
\[
\begin{split}
-{\mathcal L^2}_{\mv\mv}+\sum_{\mz\ne \mv}|{\mathcal L^2}_{\mv\mz}|&= -{\mathcal L^2}_{\mv\mv}+\sum_{\mz\in N_\mv}|{\mathcal L^2}_{\mv\mz}|+\sum_{\mz\nsim\mv}|{\mathcal L^2}_{\mv\mz}|\\
&=-\deg^2(\mv)-\deg(\mv)+\sum_{\mz\in N_\mv}(\deg(\mv)+\deg(\mz)-|N_{\mv}\cap N_{\mz}|) +\sum_{\mz\nsim\mv} |N_{\mv}\cap N_{\mz}|\\
&=-\deg(\mv)+\sum_{\mz\in N_\mv}(\deg(\mz)-|N_{\mv}\cap N_{\mz}|) +\sum_{\mz\nsim\mv} |N_{\mv}\cap N_{\mz}|\\
&\ge
-\deg(\mv)+\sum_{\mz\in N_\mv}(|N_\mz|-|N_{\mv}\cap N_{\mz}|) + 1\\
&= \sum_{\mz\in N_\mv}(|N_{\mz}\setminus N_{\mv}|-1) + 1\ge 1
\end{split}
\]
where the last inequality holds because $\mv\in N_\mz\setminus N_\mv$ for all $\mz$ adjacent to $\mv$.
(We recall that $N_{\tilde{\mv}}$ denotes the neighborhood of a vertex $\tilde{\mv}$.)
\end{proof}

However, weaker contractivity properties do hold for  any graph, whether complete or not.

\begin{prop}\label{prop:ellp-discr}
The semigroup $(e^{-t\mathcal L^2})_{t\ge 0}$ is $\ell^p$-contractive for \textit{some} $p\in (2,\infty)$.
\end{prop}

The following smart proof has been suggested to us in~\cite{Fed17}.

\begin{proof}
By Lumer--Phillips' Theorem, all we have to show is that $-\mathcal L^2$ is dissipative whenever regarded as an operator on $\ell^p(\mV)$: because the duality mapping between $\ell^p$ and its dual space $\ell^{p'}$ is $f\mapsto |f|^{p-2}f$, this amounts to saying that for some $\alpha>0$
\begin{equation}\label{eq:fedja}
(\mathcal L f,\mathcal L |f|^\alpha f)_{\ell^2(\mV)}\ge 0\qquad \hbox{for all }f\in \R^V
\end{equation}
where we have set $\alpha:=p-2$. We are going to show that this holds for $\alpha>0$ small enough. To this aim, we are going to show that 1) $\|\mathcal L f\|^2_{\ell^2(\mV)}$  is bounded from below away from $0$ for $\Vert f\Vert_{\ell^2(\mV)}=1$ and 2) $(\mathcal L f,\mathcal L |f|^\alpha f)_{\ell^2(\mV)}\gtrapprox \|\mathcal L f\|^2_{\ell^2(\mV)}$.
 
 Step 1) Assume without loss of generality that $\|\mathcal I^T f\|^2_{\ell^2(\mE)}=1$.  Because the incidence matrix satisfies $\mathcal I^T F=0$ for all constant functions $F\in \R^V$, we deduce that
\begin{equation*}\label{eq:fedmain}
(f-F,\mathcal Lf)_{\ell^2(\mV)}=
\sum\limits_{\me\in \mE}|\mathcal I^T f(\me)|^2=1\qquad \hbox{for any }F\in\mathbb R
\end{equation*}
and by Cauchy--Schwarz
\[
 \|\mathcal L f\|^2_{\ell^2(\mV)}\ge \|f-F\|_{\ell^2(\mV)}^{-2}.
\]
Furthermore, again because $\|\mathcal I^T f\|^2_{\ell^2(\mE)}=1$ the oscillation of $f$ is bounded by some $C>0$, say 
\[
|f(\mv)-F|\le C \qquad\hbox{ for all }\mv,
\]
where $F$ is the mean value of $f$, hence 
\[
\|\mathcal L f\|^2_{\ell^2(\mV)}\ge c
\]
for $c:=(C^2 V)^{-1}$.

Step 2) Let us now fix some $A>0$ and $\alpha\in (0,1)$ and consider the two cases 
\begin{itemize}
\item $F\in [-A,A]$,
\item $F\notin [-A,A]$.
\end{itemize}
In the former case we have $y|y|^\alpha\to y$ uniformly on $[-A-C,A+C]$ as $\alpha\to 0^+$ and $(\mathcal L f,\mathcal L |f|^\alpha f)_{\ell^2(\mV)}$ tends to $\|\mathcal L f\|^2_{\ell^2(\mV)}$ uniformly for $f$ with values in $[-A-C,A+C]$. 

In the latter case, writing $f:=F+g$ for $|g|\le C\ll F$ (assuming that $F\gg 0$; the case of $F\ll 0$ can be treated likewise) we obtain
\[
\begin{split}
|f(\mv)|^\alpha f(\mv)-|f(\mw)|^\alpha f(\mw)&=F^{1+\alpha}\left(1+\frac{g(\mv)}{F}\right)^{1+\alpha}-F^{1+\alpha}\left(1+\frac{g(\mw)}{F}\right)^{1+\alpha}\\
&=F^{1+\alpha}\left(\frac{(1+\alpha)(g(\mv)-g(\mw))}{F}\right)+O\left(\frac{|g(\mv)-g(\mw)|(|g(\mv)|+|g(\mw)|)}{F^{1+\alpha}}\right)\\
&=(1+\alpha)F^\alpha(g(\mv)-g(\mw))+O\left(\frac{|g(\mv)-g(\mw)|(|g(\mv)|+|g(\mw)|)}{F^{1+\alpha}}\right).
\end{split}
\]
This means that for any $\alpha \in (0,1)$ the term $F^{-\alpha}(\mathcal L|f|^\alpha f)(\mv)$ converges uniformly to $(1+\alpha)\mathcal Lf(\mv)$ as $F\to\infty$, so 
\[
(\mathcal L f,\mathcal L |f|^\alpha f)_{\ell^2(\mV)}\approx (1+\alpha)F^\alpha \|\mathcal Lf\|^2_{\ell^2(\mV)} \qquad \hbox{for large }F.
\]
Thus there exists $A$ such that for all $|F|>A$ ~\eqref{eq:fedja} holds.
\end{proof}

\begin{rem}
It follows from the Riesz--Thorin Theorem that if $\mG$ is not complete, then those $p$ such that $(e^{-t\mathcal L^2})_{t\ge 0}$ is $\ell^p$-contractive form a connected subset of $(1,\infty)$. Let us observe that for the star graph on two edges (i.e., for the path graph on 3 vertices, cf.\ Remark~\ref{rem:compl}), the transition takes place at some $p_0\le 6$: indeed, for 
\[
f=\begin{pmatrix}
1\\ \frac{7}{4}\\ -1
\end{pmatrix}
\]
one finds for
\begin{equation*}\label{eq:globest}
\kappa_f(p):=(\mathcal L f,\mathcal L |f|^{p-2}f)_{\ell^2(\mV)}
\end{equation*}
the approximate values $\kappa_f(2)=\frac{49}{8}$, $\kappa_f(5)\approx 2.31$, $\kappa_f(5.71)\approx 0.02$, $\kappa_f(5.72)\approx -0.02$. 
More generally, it can be seen that for stars on $E$ edges the lowest $p_0$ for which $\kappa_f(p_0)<0$ for some $f$ satisfies $p_0\to 2$ as $E\to \infty$. We expect the internal structure of the graph to decisively determine such a $p_0$.
\end{rem}

Since the constant function ${\bf 1}$ lies in the null space of $\mathcal L$, the lowest eigenvalue of the discrete bi-Laplacian is always $0$ -- its multiplicity being the number of connected components of $\mG$, i.e., 1 under our standing assumptions. By the Spectral Mapping Theorem, the eigenvalues of the square $A^2$ of a symmetric matrix $A$ are precisely the square of the eigenvalues of $A$: we can thus deduce the following estimate on the spectral gap of $\mathcal L^2$ from a well-known result in~\cite{Fie73}.

\begin{prop}\label{prop:eigenv-discr}
The second lowest eigenvalue $\lambda_2$ of the $\mathcal L^2$ on $\mG$ cannot be smaller than the second lowest eigenvalue $4\left(1-\cos\left(\frac{\pi}{V}\right)\right)^2$ of $\mathcal L^2$ on a path graph on $V$ vertices; and it cannot be larger than the second lowest eigenvalue $V^2$ of $\mathcal L^2$ on the complete graph on $V$ vertices. 

Accordingly, $(e^{-t\mathcal L^2})_{t\ge 0}$ converges uniformly to the orthogonal projector onto the space spanned by $\bf 1$: the rate of convergence is $e^{-\lambda_2 t}$, hence it is fastest for the complete graph on $V$ vertices and slowest for the path graph on $V$ vertices.
\end{prop}

\begin{rem}\label{rem:moh}
Let $\mG$ be an \emph{infinite} graph: if $\deg\in \ell^\infty(\mV)$, then $\mG$ is said to be \emph	{uniformly locally finite}. One can then consider the quadratic form $\mathfrak h$ defined as in~\eqref{eq:formah}, this time with maximal domain. If $\mG$ is uniformly locally finite, then 
%
 by~\cite[Thm.~3.2]{Moh82} the operator $\mathcal L$ associated with $\mathfrak h$ is bounded, self-adjoint, and positive semidefinite  on the Hilbert space $\ell^2(\mV)$, hence $-\mathcal L^2$ generates a cosine operator function $(\cos(-t\mathcal L^2))_{t\in \R}$ and an analytic, contractive semigroup $(e^{-t\mathcal L^2})_{t\ge 0}$ of angle $\frac{\pi}{2}$ on $\ell^2(\mV)$.

By~\cite[Exa.~3.14.15]{AreBatHie01} the cosine operator function can be expressed as
\begin{equation*}\label{eq:fattorini}
\cos(-t\mathcal L^2)=\frac{1}{2}\left(e^{it \mathcal L }+e^{-it\mathcal L}\right),\quad t\ge 0,
\end{equation*}
which in turn yields quite explicit analytic expressions in those cases for which $\mathcal L$ is known explicitly. By~\cite[Exa.~12.3.3]{Dav07} we obtain for instance for $f\in \ell^2(\mV)$ the expression
\begin{equation*}\label{eq:estliz}
e^{-it\mathcal L}f(\mv)=\sum_{\mw\in \Z}f(\mw) \frac{1}{2\pi}\int_{-\pi}^\pi \cos((\mv-\mw)q) e^{-2it(1-\cos q)}dq,\quad t\in\R,\ \mv\in\Z,
\end{equation*}
for the unitary group generated by the discrete Laplacian on the graph (identified with) $\Z$, hence
\begin{equation*}\label{eq:bi-estliz}
\begin{split}
\cos(-t\mathcal L^2)f(\mv)&=\sum_{\mw\in \Z}f(\mw) \frac{1}{2\pi}\int_{-\pi}^\pi \cos((\mv-\mw)q) \cos(2t(1-\cos q))dq\\
&=\sum_{k\in\Z_p}f(\mv-k)i^kJ_{|k|}(2t)\cos(2t)+\sum_{k\in\Z_d}f(\mv-k)i^{k+1}J_{|k|}(2t)\sin(2t),
\end{split}
\quad t\in\R,\ \mv\in\Z,
\end{equation*}
where $J_k$ denotes the Bessel function of the first kind.
\end{rem}

\begin{rem}
While it is not clear how to define the discrete Laplacian if a graph contains loops, the theory of discrete Laplacians is not too strongly dependent on the assumption that the graph contains no multiple edges. The definition of the discrete Laplacian as $\mathcal L:=\mathcal I\mathcal I^T$ does still make sense, and we can still consider the self-adjoint, positive semi-definite operator $\mathcal L^2$; accordingly, if $\mG$ contains multiple edges Propositions~\ref{prop:discr-wellp} still holds, and one can check that also the proof of Proposition~\ref{prop:ellp-discr} goes through. However, the simple example of the path graph on three vertices show that~\eqref{eq:bilapldiscr} fails to hold in this case, and therefore the proof of Proposition~\ref{prop-complete} is affected, too.
\end{rem}

\section{Bi-Laplacians on networks: General setting and Self-adjoint extensions}\label{sec:gener}

Let us now turn to the main topic of this paper, that is, bi-Laplacians on \textit{networks}/\textit{metric graphs}: we come back to the setting of Section~\ref{sec:discrete} and consider a finite connected graph $\mG=(\mV,\mE)$ without loops or multiple edges, with $V:=|\mV|$ and $E:=|\mE|$. 
We also denote by $\mE_{\mv}$ the set of all edges $\mv$ is incident with. Clearly $|\mE_{\mv}|=\deg(\mv)$ for all $\mv\in \mV$. We fix an arbitrary orientation of $\mG$, so that each edge $\me\equiv(\mv,\mw)$ can be identified with an interval $[0,\ell_\me]$ and its endpoints $\mv,\mw$ with $0$ and $\ell_\me$, respectively. In such a way one naturally turns the $\mG$ into a metric measure space $\mathcal G$: a metric graph whose underlying discrete graph is precisely $\mG$.

We regard functions on $\mathcal G$ as vectors $(u_\me)_{\me\in\mE}$, where each $u_\me$ is defined on the edge $\me\simeq(0,\ell_\me)$. As usual in the literature, we introduce the Hilbert space of measurable, square integrable functions on $\mathcal G$
\[L^2(\mathcal G):=\bigoplus_{\me\in \mE} L^2(0,\ell_\me)=\left\{u=(u_\me)_{\me\in\mE} \textrm{ s.t. }\, u_\me:(0,\ell_\me)\to\C\ \textrm{is measurable and }  \sum\limits_{\me\in\mE} \int_0^{\ell_\me}|u_\me(x)|^2\,dx<\infty\right\}, \]
endowed with the natural inner product \[(u,v)_{L^2(\mathcal G)}:=\sum\limits_{\me\in\mE}\int_0^{\ell_\me} u_\me(x)\overline{v_\me(x)}\,dx.\]
Our approach to the study of bi-Laplacians is based on sesquilinear form methods. To this aim, we introduce the Sobolev space
\[\h^k(\mathcal G):=\bigoplus_{\me\in\mE} H^k(0,\ell_\me),\qquad k\in\mathbb N,\]
consisting of functions supported on the edges whose first $k$-th weak derivatives are square integrable. Let us stress that the connectivity of $\mathcal G$ does not affect the definition of such Sobolev spaces. Even though boundary values are well-defined for elements of $\h^k(\mathcal G)$, only later on will we describe them in terms of transmission conditions in the vertices of $\mathcal G$.

We are going to study the bi-Laplacian
$$A:u\mapsto u''''$$ acting on each edge of  $\mathcal G$.
In order to completely define $A$ on $L^2(\mathcal{G})$ one has to specify its domain: we can think of its realizations $A_0$ and $A_{\max}$ with minimal and maximal domains
\[
D(A_0):=\bigoplus_{\me\in\mE} H^4_0(0,\ell_\me),\quad D(A_{\max}):=\tilde{H^4}(\mathcal G)=\bigoplus_{\me\in\mE} H^4(0,\ell_\me),
\]
respectively. 
Both $A_0$ and $A_{\max}$ are closed and semi-bounded, but they do not generate strongly continuous semigroups on $L^2(\mathcal G)$; in between there are infinitely many other realizations, whose generation property will depend on conditions on the boundary values of functions in $\tilde{H}^4(\mathcal G)$ and their derivatives of order one, two, three: these boundary values are well-defined since $\tilde{H}^4(\mathcal G)\hookrightarrow \bigoplus_{\me\in\mE}C^3([0,\ell_\me])$. 

In the search for realizations of $A$ that are semigroup generators, let us first search for those extensions of $A_0$ that are self-adjoint on $L^2(\mathcal G)$: they certainly exist, since $A_0$ is a symmetric, positive semidefinite operator and we can therefore apply classic extension theory, see e.g.\ the overview in~\cite{FukOshTak10,Sch12}.  Let us describe them systematically:
integrating by parts we find for all $u,v\in \h^4(\mathcal G)$
\begin{equation*}
\begin{split}
(A u,v)_{L^2(\mathcal G)}&=\sum\limits_{\me\in\mE}\int_0^{\ell_\me} u_\me^{''''}(x)\overline{v_\me(x)}\,dx\\
&=\sum\limits_{\me\in\mE}\left[u_\me^{'''}\overline{v_\me}\right]_0^{\ell_\me}-\sum\limits_{\me\in\mE}\left[u_\me^{''}\overline{v'_\me}\right]_0^{\ell_\me}+\sum\limits_{\me\in\mE}\left[u'_\me \overline{v^{''}_\me}\right]_0^{\ell_\me}\\
&\quad -\sum\limits_{\me\in\mE}\left[u_\me \overline{v^{'''}_{\me}}\right]_0^{\ell_\me}+\sum\limits_{\me\in\mE}\int_0^{\ell_\me} u_\me(x)\overline{v^{''''}_\me(x)}\,dx\ .
\end{split}
\end{equation*}
Adopting the notations $u(0):=(u_\me(0))_{\me\in \mE}$ and  $u(\ell):=(u_\me(\ell_\me))_{\me\in\mE}$,  we find that $(A u,v)_{L^2(\mathcal G)}=(u,Av)_{L^2(\mathcal G)}$ -- and hence $A$ is symmetric -- if and only if 
\begin{equation}\label{selfadj-del2}
\left(\begin{pmatrix} u(0)\\u(\ell)\\-u'(0)\\u'(\ell)\end{pmatrix},\begin{pmatrix} -v'''(0)\\v'''(\ell)\\-v''(0)\\-v''(\ell)\end{pmatrix}\right)_{\C^{4E}}=\left(\begin{pmatrix} -u'''(0)\\u'''(\ell)\\-u''(0)\\-u''(\ell)\end{pmatrix},\begin{pmatrix} v(0)\\v(\ell)\\-v'(0)\\v'(\ell)\end{pmatrix}\right)_{\C^{4E}}\hbox{for all }u,v\in D(A):
\end{equation}

If $Y$ is a 
subspace of $\C^{4E}$, then a sufficient condition for \eqref{selfadj-del2} to hold is that
\begin{equation*}
\begin{pmatrix} u(0)\\u(\ell)\\-u'(0)\\u'(\ell)\end{pmatrix},\begin{pmatrix} v(0)\\v(\ell)\\-v'(0)\\v'(\ell)\end{pmatrix}\in Y\quad \textrm{and}\quad \begin{pmatrix} -u'''(0)\\u'''(\ell)\\-u''(0)\\-u''(\ell)\end{pmatrix},\begin{pmatrix} -v'''(0)\\v'''(\ell)\\-v''(0)\\-v''(\ell)\end{pmatrix}\in Y^\perp\ .
\end{equation*}
This boundary condition can be further generalized considering some  $R\in\mathcal{L}(Y)$ and imposing
\begin{equation}\label{selfadj-del3}
 \begin{pmatrix} u(0)\\u(\ell)\\-u'(0)\\u'(\ell)\end{pmatrix}\in Y,\qquad \begin{pmatrix} -u'''(0)\\u'''(\ell)\\-u''(0)\\-u''(\ell)\end{pmatrix}+R\begin{pmatrix} u(0)\\u(\ell)\\-u'(0)\\u'(\ell)\end{pmatrix}\in Y^\perp
 \end{equation}
on all $u$ in the domain of $A$.

This kind of parametrization of boundary conditions for one-dimensional operators is not completely new in literature: in the case of Laplacians on networks, comparable conditions are discussed for example in~\cite{Kuc04}, but they go back at least to~\cite{Hol39} in the context of Sturm--Liouville problems. 
In particular, in \cite{Kuc04} a characterization of those conditions that make the Laplacian self-adjoint is provided relating a parametrization similar to~\eqref{selfadj-del3} to the earlier one discussed in~\cite{KosSch99}.
Following the same strategy  we obtain the following  theorem. In order to shorten the notation we will denote the trace vectors as
\[{\bf u}^{0,1}:=\begin{pmatrix} u(0)\\u(\ell)\\-u'(0)\\u'(\ell)\end{pmatrix},\qquad {\bf u}^{3,2}:=\begin{pmatrix} -u'''(0)\\u'''(\ell)\\-u''(0)\\-u''(\ell)\end{pmatrix}.\]

\begin{theorem}\label{thm:charly}
For any realization $A_0\subset A\subset A_{\max}$, the following assertions are equivalent:
\begin{enumerate}[(i)]
\item $A$ is self-adjoint on $L^2(\mathcal G)$;
\item the vertex conditions can be written as
\begin{equation}\label{condR} 
{\bf u}^{0,1}\in Y,\qquad {\bf u}^{3,2}+R
{\bf u}^{0,1} \in Y^\perp
\end{equation}
where $Y$ is a subspace of $\C^{4E}$ and $R\in\mathcal{L}(Y)$ is  self-adjoint;
\item the vertex conditions can be written as
\begin{equation}\label{condAB} C{\bf u}^{0,1}+B{\bf u}^{3,2}=0\end{equation}
for  $C,B$ $4E\times 4E$-matrices on $\C$ such that the $4E\times 8E$-matrix $(CB)$ has maximal rank and the $4E\times 4E$-matrix
$CB^{*}$ is Hermitian.
\end{enumerate}
\end{theorem}

\begin{proof}
(i)$\Rightarrow$(ii) As observed before, self-adjointness of $A$ implies \eqref{selfadj-del2}. Let $u\in \h^4(\mathcal G)$ satisfying \eqref{condR} and $v\in \bigoplus_{\me\in\mE} C^\infty(0,\ell_\me)$. Then, $v$ satisfies \eqref{condR}, too. Indeed, condition \eqref{condR} says that the vectors ${\bf u}^{3,2}+R{\bf u}^{0,1}$ and ${\bf u}^{0,1}$ are orthogonal, which implies ${\bf u}^{3,2}+R^*{\bf u}^{0,1}$ and ${\bf u}^{0,1}$ are orthogonal, too. In particular, one has ${\bf u}^{3,2}_Y+R^*{\bf u}^{0,1}_Y=0$. Moreover, \eqref{selfadj-del2} holds on both $Y^\perp$ and $Y$. Then, for the first one, one has
\[0=\left({\bf u}^{3,2},{\bf v}^{0,1}\right)_{Y^\perp},\]
hence ${\bf v}^{0,1}\in Y$. For the second one
\begin{align*}
\left({\bf u}^{0,1},{\bf v}^{3,2}\right)_Y&=\left({\bf u}^{3,2},{\bf v}^{0,1}\right)_Y=-\left(R^*{\bf u}^{0,1},{\bf v}^{0,1}\right)_Y=-\left({\bf u}^{0,1},R{\bf v}^{0,1}\right)_Y
\end{align*}
so that 
\[\left({\bf u}^{0,1},{\bf v}^{3,2} +R{\bf v}^{0,1}\right)_Y=0\]
and
\[{\bf v}^{3,2}+R{\bf v}^{0,1}\in Y^\perp.\]

Now, since $u$ and $v$ satisfy \eqref{condR} and \eqref{selfadj-del2} holds, one has
\[\left({\bf u}^{0,1},{\bf v}^{3,2}\right)_Y=\left({\bf u}^{3,2},{\bf v}^{0,1}\right)_Y\]
i.e.,
\[\left({\bf u}^{0,1},R{\bf v}^{0,1}\right)_Y=\left(R{\bf u}^{0,1},{\bf v}^{0,1}\right)_Y.\]
 Thus, $R$ is self-adjoint.

(ii)$\Rightarrow$(iii) Choose $C$ and $B$ such that $Y=\,$Rg$\,B^*$, $Y^\perp=\ker B$  and $R=(Q_1BQ)^{-1}C$ where $Q$ and $Q_1$ are the orthogonal projections onto Rg $B^*$ and Rg $B$ respectively. Here Rg denotes the range of a matrix. Observe that the mapping $Q_1BQ:$ Rg $B^*\to$  Rg $B$ is invertible follows from the decomposition $\C^{4E}=\ker B\oplus$ Rg $B^*$. Indeed, $(Q_1BQ)h=0$ for some $h\in$ Rg $B^*$ implies $Bh=0$ and then $h\in\ker B$. Hence, the only possibility is $h=0$. Furthermore, since $R$ is self-adjoint $(Q_1BQ)^{-1}C$ is self-adjoint, which implies $B^{-1}C$ is self-adjoint, too, i.e., $B^{-1}C=C^*{(B^*)}^{-1}$. Multiplying by appropriate inverse matrices from both sides we get  $CB^*$ is self-adjoint. Then, since $A$ acts as the fourth derivative on each edge, one needs to establish four boundary conditions per edge. Therefore for a function in $D(A)$ we need exactly $4E$ conditions. This is assured with \eqref{condR} because $Y$ is a subspace of $\C^{4E}$. Thus, in order to impose the right number of conditions with \eqref{condAB}, the rank of the matrix $(CB)$ should be maximal.  Finally,  condition \eqref{condR} says in particular that 
\[{\bf u}^{0,1}\in\ \textrm{Rg}\ B^*\ \textrm{and}\ {\bf u}^{3,2}+(Q_1BQ)^{-1}C{\bf u}^{0,1}\in\ker B,\]
which is, denoting by $P$ the orthogonal projection onto $\ker B$,
\[P{\bf u}^{0,1}=0\ \textrm{and}\ Q{\bf u}^{3,2}+(Q_1BQ)^{-1}CQ{\bf u}^{0,1}=0.\]
By \cite[Cor.~5]{Kuc04}, this is equivalent to
\eqref{condAB}. 

(iii)$\Rightarrow$(i)
 In order to prove self-adjointness of $A$ we have to establish two facts: (a) if $u$ and $v$ satisfy \eqref{condAB}, then \eqref{selfadj-del2} holds and (b) if $u$ satisfies \eqref{condAB} and \eqref{selfadj-del2} holds, then $v$ also satisfies \eqref{condAB}. For the first assertion, consider 
\[\left({\bf u}^{0,1},{\bf v}^{3,2}\right)=-\left({\bf u}^{0,1},B^{-1}C{\bf v}^{0,1}\right)\]
and similarly
\[\left({\bf u}^{3,2},{\bf v}^{0,1}\right)=-\left(B^{-1}C{\bf u}^{0,1},{\bf v}^{0,1}\right).\] 
Then, since $B^{-1}C$ is self-adjoint, one obtains 
\[\left({\bf u}^{0,1},{\bf v}^{3,2}\right)=\left({\bf u}^{3,2},{\bf v}^{0,1}\right)\]
which yields \eqref{selfadj-del2}.

For the second assertion, since $u$ satisfies \eqref{condAB} and $CB^*$ is self-adjoint, then for any $h\in\C^{4E}$ one can write 
\[{\bf u}^{0,1}=-B^*h\ \textrm{and}\ {\bf u}^{3,2}=C^*h.\]
Substituting in \eqref{selfadj-del2} one obtains
\[0=\left(C^*h, {\bf v}^{0,1}\right)-\left(-B^*h,{\bf v}^{3,2}\right)=\left(h,C{\bf v}^{0,1}+B{\bf v}^{3,2}\right).\]
Since $h$ is arbitrary, it follows 
\[C{\bf v}^{0,1}+B{\bf v}^{3,2}=0\]
and then $v$ satisfies \eqref{condAB}.  
\end{proof}

Let us stress that the content of Theorem~\ref{thm:charly} is simply an assertion about extensions of vector-valued functions and is not yet taking into account the structure of the underlying network $\mathcal G$.

If local models (say, networks of beams) are considered, it is reasonable to endow operators with vertex interactions that mirror the graph's connectivity by only involving boundary values on adjacent edges.  This is made as follows: once assigned an arbitrary direction to each edge of $\mG$, the entries of any vector in $\C^{2E}$ are bijective with the boundary values $\begin{pmatrix}(\psi_\me(0))_{\me\in\mE}\\(\psi_\me(\ell))_{\me\in\mE}\end{pmatrix}$ of a function $\psi\in \tilde{H}^1(\mathcal G)$ defined on each edge. Let us consider the subspace $c_\mV$ of $\C^{2E}$ that consists of those vectors that are vertex-wise constant: i.e., if a vertex $\mv$ is incident with edges $\me_{j_1},\ldots,\me_{j_n}$, then elements of $c_\mV$ have entries that agree whenever they are associated with the endpoints of $\me_{j_1},\ldots,\me_{j_n}$ corresponding to $\mv$ (in the case of a star consisting of $k$ semi-infinite intervals, say, $c_\mV$ would e.g.\ be spanned by $1_{\R^k}=(1,\ldots,1)$). Hence, imposing the condition that 
\[
\begin{pmatrix}
u(0)\\ u(\ell)
\end{pmatrix}\in c_\mV
\]
is equivalent to require that a function is \textit{continuous in all vertices}, i.e., it attains the same value  on all interval endpoints corresponding to the same vertex.



Before describing a few instances of vertex conditions with interesting properties, we fix some notations regarding vertex conditions. Here and in the following we denote by $\frac{\partial u_\me	}{\partial{\nu}}(\mv)$ the exterior normal derivative of $u_\me$ at $\mv$, i.e., 
\[
\frac{\partial u_\me	}{\partial{\nu}}(\mv):=
\begin{cases}-u'_\me(0)\quad &\hbox{if $\mv$ is the initial endpoint of $\me$}\\
u'_\me(\ell_\me)&\hbox{if $\mv$ is the terminal endpoint of $\me$}.
\end{cases}
\]
Likewise, $\frac{\partial^3 u_\me}{\partial{\nu}^3}(\mv)$ denotes $-u'''_\me(0)$ or $u'''_\me(\ell_\me)$, respectively.

 In analogy with Kirchhoff's classical current law (``In a network of conductors meeting at a point, the algebraic sum of currents  is zero''), the so-called \textit{Kirchhoff conditions} impose that the entries of a vector of current-like values sum up to 0 at each vector of a network. This terminology is rather common in the theory of Laplacians on  networks since~\cite{Bel85,KosSch99} and is canonically referred to the vector of normal derivatives (at each vertex), i.e., it requires that 
 \[
 \sum\limits_{\me\in\mE_\mv}\frac{\partial u_\me}{\partial\nu}(\mv)=0\qquad\hbox{for all }\mv\in\mV,
 \]
 where $\mE_\mv$ is as usual the set of all edges $\mv$ is incident with.
 In our context we will also occasionally consider Kirchhoff-type conditions on the boundary values $u''_\me(\mv)$ of the second derivatives as well as on the third normal derivatives $\frac{\partial u^3_\me}{\partial\nu^3}(\mv)$. 

Let us present an overview of vertex conditions that have appeared in the literature on beam equations on networks, and how they can be discussed with our formalism.}

\begin{ex}\label{ex1} Dekoninck and Nicaise have studied in \cite{DekNic99} the exact controllability problem of
networks of beams
\[
\frac{\partial^2 u_\me}{\partial t^2}(t,x)=-\frac{\partial^4 u_\me}{\partial x^4}(t,x),\qquad t\ge 0,\ \me\in \mE,\ x\in (0,\ell_\me)
\]
under vertex conditions

\begin{equation}\label{condNic1}
\begin{cases}
u_\me(t,\mv)=u_\mf(t,\mv)& \textrm{if}\ \me\cap \mf=\mv,\ t\ge 0,\\
\sum\limits_{\me\in \mE_{\mv}}\frac{\partial u_\me}{\partial \nu}(t,\mv)=0& \forall\ \mv\in \mV,\ t\ge 0,\\
\frac{\partial^2 u_\me}{\partial x^2}(t,\mv)=\frac{\partial^2 u_\mf}{\partial x^2}(t,\mv)& \textrm{if}\ \me\cap \mf=\mv,\ t\ge 0,\\
\sum\limits_{\me\in \mE_{\mv}}	\frac{\partial^3 u_\me}{\partial \nu^3}(t,\mv)=0& \forall\ \mv\in \mV,\ t\ge 0,
\end{cases}
\end{equation}
that is, they require continuity of the function and of the second derivatives at each vertex of $\mathcal G$, and  Kirchhoff condition for the normal  derivatives and third normal derivatives. 
Condition \eqref{condNic1} is equivalent to~\eqref{condR} with
\[Y=\cmv\times \cmv^\perp,\quad R=0.\]
Indeed, one obtains 
\[
\begin{pmatrix}u(0)\\u(\ell)\end{pmatrix}\in\cmv\quad\hbox{and}\quad \begin{pmatrix}-u'(0)\\u'(\ell)\end{pmatrix}\in\cmv^\perp,
\]
which yield continuity of the function at the vertices and the Kirchhoff condition on the normal derivatives, respectively. Consequently, the additional conditions
\[
\begin{pmatrix}-u''(0)\\-u''(\ell)\end{pmatrix}\in\cmv\quad\hbox{and}\quad \begin{pmatrix}-u'''(0)\\u'''(\ell)\end{pmatrix}\in\cmv^\perp,
\]
hold, meaning continuity and Kirchhoff conditions on the second derivatives and third normal derivatives, respectively. 

Observe that the bi-Laplacian whose domain consist of $\h^4(\mathcal G)$-functions satisfying \eqref{condNic1} is the square of the Laplacian $\Delta_{CK}$ with continuity and Kirchhoff conditions  on the normal derivatives.  Such $\Delta_{CK}$ is a favorite object in the theory of operator on networks since~\cite{Rot83,Bel85,Nic86} and is often regarded as the natural counterpart of the Neumann Laplacian on domains. Accordingly, conditions~\eqref{condNic1} are the network analogue of sliding boundary conditions for bi-Laplacians on domains.
\end{ex}

\begin{ex}\label{ex2} Again Dekoninck and Nicaise have studied in \cite{DekNic00} the characteristic equation for the spectrum of the operator acting on functions that satisfy vertex conditions either
\begin{equation}\label{condNic2}
\begin{cases}
u_\me(t,\mv)=u_\mf(t,\mv)& \textrm{if}\ \me\cap\mf =\mv,\ t\ge 0,\\
\frac{\partial u_\me}{\partial\nu}(t,\mv)=\frac{\partial u_\mf}{\partial\nu}(t,\mv)& \textrm{if}\ \me\cap\mf =\mv,\ t\ge 0,\\
\sum\limits_{\me \in \mE_{\mv}}\frac{\partial^2 u_\me}{\partial x^2}(t,\mv)=0& \forall\ \mv\in \mV,\ t\ge 0,\\
\sum\limits_{\me\in \mE_{\mv}}\frac{\partial^3 u_\me}{\partial \nu^3}(t,\mv)=0& \forall\ \mv\in \mV,\ t\ge 0,
\end{cases}
\end{equation}
(analogous to conditions considered in~\cite{CheDelKra87})
or
\begin{equation}\label{condNic3}
\begin{cases}
u_\me(t,\mv)=u_\mf(t,\mv)& \textrm{if}\ \me\cap\mf =\mv,\ t\ge 0,\\
\frac{\partial^2 u_\me}{\partial x^2}(t,\mv)=0 &\forall\ \mv\in \mV,\ t\ge 0,\\
\sum\limits_{\me\in \mE_{\mv}}\frac{\partial^3 u_\me}{\partial \nu^3}(\mv)=0& \forall\ \mv\in \mV,\ t\ge 0.
\end{cases}
\end{equation} 

Condition \eqref{condNic2} can be represented as~\eqref{condR} upon taking
\[Y=\cmv\times \cmv,\quad R=0,\]
whereas condition \eqref{condNic3} corresponds to \[Y=\cmv\times \C^{2E},\quad R=0;\]
 the latter has been  generalized in~\cite{BorLaz04} by adding lower order terms, corresponding to a (possibly non-self-adjoint) term $R\ne 0$.
\end{ex}

\begin{ex}\label{ex3}
Kiik, Kurasov and Usman discuss in \cite{KiiKurUsm15} vertex conditions which take into account the geometry of the underlying graph and depend on the angles $\alpha$ (resp., $\beta,\gamma$) between the edges $\me_2,\me_3$ (resp., between $\me_1,\me_3$ and $\me_1,\me_2$). In the rather special setting of a star consisting of three semi-infinite edges they rigorously derive from physical-geometrical considerations the vertex conditions
\begin{equation}\label{condKiik}
\begin{cases}
u_1(t,0)=u_2(t,0)=u_3(t,0),\\ 
\sin\alpha\cdot \frac{\partial u_1}{\partial x}(t,0)+\sin\beta\cdot \frac{\partial u_2}{\partial x}(t,0)+\sin\gamma \cdot \frac{\partial u_3}{\partial x}(t,0)=0,\\ 
\frac{1}{\sin\alpha}\frac{\partial^2 u_1}{\partial x^2}(t,0)=\frac{1}{\sin\beta}\frac{\partial^2 u_2}{\partial x^2}(t,0)=\frac{1}{\sin\gamma}\frac{\partial^2 u_3}{\partial x^2}(t,0),\\ 
\frac{\partial^3 u_1}{\partial x^3}(t,0)+\frac{\partial^3 u_2}{\partial x^3}(t,0)+\frac{\partial^3 u_3}{\partial x^3}(t,0)=0,
\end{cases}
\end{equation}
provided none of the angles $\alpha, \beta$, and $\gamma$ is equal to 0 or $\pi$; the same conditions were mentioned, without rigorous derivation, already in~\cite{BorLaz04}. (In~\cite{LagLeuSch92,LagLeuSch93}, comparable but physically more realistic derivations lead to various systems of nonlinear equations.)

%

Conditions \eqref{condKiik} can be represented in our standard formalism by taking 
\[
Y=Y_1\times Y_2
\]
with
\[
Y_1=\langle\begin{pmatrix}1\\1\\1\end{pmatrix}\rangle,\qquad   Y_2=\langle\begin{pmatrix}\sin\alpha\\ \sin\beta\\ \sin\gamma\end{pmatrix}\rangle^\perp,\quad \hbox{and}\quad  R=0.\] 
\end{ex}

Motivated by modelling of mechanical systems, further realizations with \textit{dynamic} transmission conditions in the vertices have been studied among others in~\cite{CheDelKra87,LagLeuSch94}. We will discuss such dynamic conditions in a forthcoming paper.

\section{Generation results in $L^2$}\label{sec:selfa}
Let us now turn to the issue of well-posedness of the parabolic equation associated with different realizations of the bi-Laplacian $A$ on $L^2(\mathcal G)$. We are going to focus on the parametrization in~\eqref{condR}, i.e., we consider the edgewise fourth derivative
\[
A:u\mapsto u''''
\]
with domain
\begin{align*}
D(A)&=\Bigg\{u\in \h^4(\mathcal G) : {\bf u}^{0,1}\in Y\, \textrm{and}\, {\bf u}^{3,2}+R{\bf u}^{0,1}\in Y^\perp \Bigg\}.
\end{align*}

 (If there is any risk of confusion, we will stress the dependence of $A$ on the boundary conditions in~\eqref{condR} by writing $A_{Y,R}$.)

Here and in the remainder of this article, we always impose the following.

\begin{ass}\label{ass:standing}
$Y$  is a subspace of $\C^{4E}$ and $R$ is a (bounded, not necessarily self-adjoint) linear operator on $Y$.
\end{ass}

The following trace estimates will be useful in the following.
\begin{lemma}\label{lem:trace}
Let $u\in \h^4(\mathcal G)$, then 
\begin{alignat*}{4}
&|u(0)|^2&&\leq\frac{2}{q}\Vert u\Vert_{L^2(\mathcal G)}^2+q\Vert u'\Vert_{L^2(\mathcal G)}^2, \qquad\ &&|u(\ell)|^2&&\leq\frac{2}{r}\Vert u\Vert_{L^2(\mathcal G)}^2+r\Vert u'\Vert_{L^2(\mathcal G)}^2,\\
&|u'(0)|^2&&\leq\frac{2}{q}\Vert u'\Vert_{L^2(\mathcal G)}^2+q\Vert u''\Vert_{L^2(\mathcal G)}^2, &&|u'(\ell)|^2&&\leq\frac{2}{r}\Vert u'\Vert_{L^2(\mathcal G)}^2+r\Vert u''\Vert_{L^2(\mathcal G)}^2,\\
&|u''(0)|^2&&\leq\frac{2}{q}\Vert u''\Vert_{L^2(\mathcal G)}^2+q\Vert u'''\Vert_{L^2(\mathcal G)}^2,&&|u''(\ell)|^2&&\leq\frac{2}{r}\Vert u'' \Vert_{L^2(\mathcal G)}^2+r\Vert u'''\Vert_{L^2(\mathcal G)}^2,\\
&|u'''(0)|^2&&\leq\frac{2}{q}\Vert u''' \Vert_{L^2(\mathcal G)}^2+q\Vert u''''  \Vert_{L^2(\mathcal G)}^2,  &&|u'''(\ell)|^2&&\leq\frac{2}{r}\Vert u'''\Vert_{L^2(\mathcal G)}^2+r\Vert u''''  \Vert_{L^2(\mathcal G)}^2,
\end{alignat*}
for any positive $q,r\leq\min_{\me\in \mE}\ell_\me$.
\end{lemma}

\begin{proof}
Using the fundamental theorem of calculus it is easy to see that
\begin{equation*}
|u_\me(0)|^2\leq\frac{2}{q}\Vert u_\me\Vert_{L^2[0,\ell_\me]}^2+q\Vert u_\me'\Vert_{L^2[0,\ell_\me]}^2\qquad  \hbox{for all }u\in \h^4(\mathcal G)\hbox{ and }q\in (0, 1];
\end{equation*}
the other inequalities can be checked likewise.
\end{proof}

Our favorite parametrization in~\eqref{condR} allows for a slick application of classical methods based on elliptic sesquilinear forms as presented, e.g., in~\cite{DauLio88,Are06,Mug14}. 

\begin{theorem}\label{thm:forma}
The sesquilinear form associated with $A$ is given by
\begin{equation}\label{eq:formmain}
\f(u,v)=\sum\limits_{\me\in\mE}\int_0^{\ell_\me}u_\me''(x)\overline{v_\me''(x)}\,dx-\left(R {\bf u}^{0,1}, {\bf v}^{0,1}\right)_{\mathbb C^{4E}}.
\end{equation}
with domain
\begin{equation}\label{eq:formdomain}
D(\f):=\h^2_Y(\mathcal G):=\Bigg\{u\in \h^2(\mathcal G) : {\bf u}^{0,1}\in Y\Bigg\}.
\end{equation}
\end{theorem}

\begin{proof}
The first step is to prove that $\f$ is closed. Thanks to estimates in Lemma \ref{lem:trace} and the fact that $R$ is bounded it is possible to find a positive constant $C$ such that
$\Vert u\Vert_\f\leq C\Vert u\Vert_{\h^2(\mathcal G)}$ for any $u\in D(\f)$. For the same reasons one can estimate as follows
\begin{align*}
c_1\Vert u\Vert_{\h^2(\mathcal G)}^2&\leq c_2\left(\Vert u\Vert^2_{L^2(\mathcal G)}+\Vert u'\Vert_{L^2(\mathcal G)}^2+\Vert u''\Vert^2_{L^2(\mathcal G)}\right)-\Vert R\Vert_{\bound(Y)}\Vert u\Vert^2_{L^2(\partial\mathcal G)}\\
&\leq c_3 \left(\Vert u\Vert^2_{L^2(\mathcal G)}+\Vert u''\Vert_{L^2(\mathcal G)}^2\right)-\left(R {\bf u}^{0,1}, {\bf u}^{0,1}\right)\\
&\leq c_4\Vert u\Vert_\f^2
\end{align*}
for suitable positive constants such that  $c_2>\Vert R\Vert_{\bound(Y)}\left(2r+\frac{4}{r}\right)$ for $0<r\leq\min_{\me\in \mE}\ell_\me$ and for any $u\in \bigoplus_{\me\in\mE}C_c^\infty(0,\ell_\me)$.
By density, we can thus find two  constants $c,C\geq0$ such that
\begin{equation}\label{norm}c\Vert u\Vert_{\h^2(\mathcal G)}\leq\Vert u\Vert_\f\leq C\Vert u\Vert_{\h^2(\mathcal G)}\qquad \hbox{for all } u\in D(\f),\end{equation}
i.e., the norm associated with $\f$ and the norm of the space $\h^2(\mathcal G)$ are  equivalent, so that $(D(\f),\Vert\cdot\Vert_\f)$ is complete. 
 Thus, it is associated with a closed operator $S$ in $L^2(\mathcal G)$ defined by
\[
\begin{split}
 D(S)&:=\{u\in D(\f):\exists v\in L^2(\mathcal G)\ \textrm{s.t.}\ \f(u,h)=(v,h)\ \textrm{for all}\ h\in D(\f)\},\\
Su&:=v. 
 \end{split}
 \]
We want to show that $S$ coincides with $A$. So, according to the above definition, for any $u\in D(S)\subset D(\f)$ there exists $Su:=v\in L^2(\mathcal G)$ such that 
\begin{equation}\label{pedro} \f(u,h)-\sum\limits_{\me\in\mE}\int_0^{\ell_\me} v_\me(x)\overline{h_\me(x)}\,dx=0\end{equation}
for any  $h\in D(\f)$. Let $h\in D(\f)$ be any smooth function on the edges which vanishes in a neighborhood of each vertex together with its first derivative.
 Then, plugging such $h$ into \eqref{pedro} and integrating by parts, one obtains that
\[v=Su=\frac{d^4u}{dx^4}.\]
In particular, since $v\in L^2(\mathcal G)$ this also implies that $u\in \h^4(\mathcal G)$. Since $u\in D(S)$,  condition 
${\bf u}^{0,1}\in Y$
is satisfied. We have to show the remaining condition in \eqref{condR}. Then, choosing now as $h$ a  non-zero function in a neighborhood of the vertices and integrating by parts in \eqref{pedro}, leads to 
\[\sum\limits_{\me\in\mE}\left[u''_\me(x)\overline{h'_\me(x)}\right]_0^{\ell_\me}-\sum\limits_{\me\in\mE}\left[u'''_\me(x)\overline{h_\me(x)}\right]_0^{\ell_\me}-\left(R {\bf u}^{0,1}, {\bf h}^{0,1}\right)=0,\]
or equivalently
\[\left({\bf u}^{3,2}+R{\bf u}^{0,1}, {\bf h}^{0,1}\right)=0.\]
Since $h$ is an arbitrary vector satisfying  ${\bf h}^{0,1}\in Y$,
the above equality forces $u$ to satisfy 
${\bf u}^{3,2}+R{\bf u}^{0,1}\in Y^\perp$. Then, $u\in D(A)$. 

Conversely, take $u\in D(A)\subset D(\f)$. Then,  for any $h\in D(\f)$
\begin{align*}
\f(u,h)&=\sum\limits_{\me\in\mE}\int_0^{\ell_\me}u_\me''(x)\overline{h_\me''(x)}\,dx-\left(R {\bf u}^{0,1}, {\bf h}^{0,1}\right)\\
&=(A u,h)+\sum\limits_{\me\in\mE}\left[u_\me''(x)\overline{h_\me'(x)}\right]_0^{\ell_\me}-\sum\limits_{\me\in\mE}\left[u_\me'''(x)\overline{h_\me(x)}\right]_0^{\ell_\me}-\left(R {\bf u}^{0,1}, {\bf h}^{0,1}\right)_Y\\
&=(A u,h)-\left({\bf u}^{3,2}, {\bf h}^{0,1}\right)+\left({\bf u}^{3,2}, {\bf h}^{0,1}\right)\\
&=(A u,h).
\end{align*}
Then, $u\in D(S)$. This shows that the operators $A$ and $S$ coincide.
\end{proof}

 Let us further investigate on some special self-adjoint extensions of the operator $A$ making use of the associated form $\f$. 
As our overview in the Examples~\ref{ex1},~\ref{ex2}, and~\ref{ex3} has substantiated, literature on bi-Laplacians on graphs has always focused on models where functions are continuous in the vertices.  Hence, we are mainly interested in extensions  of $A_0$ -- the bi-Laplacian $A$ with minimal domain -- whose domains enforce continuity in the vertices; i.e., in realizations $A_{Y,R}$ such that $Y$ is of the form 
\begin{equation}\label{cont}
Y=\cmv\times Y_2\end{equation} 
for some subspace $Y_2$ of $\C^{2E}$.
How large is this class? 
The following is a classical result obtained in~\cite{Kre47,AndNis70}, cf.~\cite[Chapter~13]{Sch12}.

\begin{theorem}\label{thm:Kre47}
Let $A$ be a symmetric, positive semidefinite operator on a Hilbert space $H$. Then $A$ has a self-adjoint extension if and only if the associated quadratic form is densely defined in $H$.

In this case, there exist precisely two extensions $A_K,A_F$ of $A$ such that
\begin{itemize}
\item $A_K,A_F$ are self-adjoint and positive semidefinite and
\item any other self-adjoint, positive semidefinite extension $\tilde{A}$ of $A$ satisfies
\[
A_K\le \tilde{A}\le A_F,
\]
i.e., the corresponding quadratic forms $\f_K,\tilde{\f},\f_F$ satisfy
\[
D(\f_F)\subset D(\tilde{\f})\subset D(\f_K)
\]
and
\[
\f_F(u,u)\le \tilde{\f}(u,u)\le \f_K(u,u)\qquad \hbox{for all }u\in D(\f_F).
\]
\end{itemize}
\end{theorem}
The operators $A_K$ and $A_F$ are called  \emph{Krein--von Neumann} and \emph{Friedrichs extension} of $A$, respectively: with respect to the partial order just introduced, they are the smallest and the largest extensions, respectively. 

\begin{ex}\label{ex-nickur}
When looking for self-adjoint realizations of the bi-Laplacian it seems natural to restrict our maximal operator $A_{\max}$  to the  space
\[
D_{cont}:=\{f\in C(\mathcal G) : f'\in \bigoplus C^\infty_c(0,\ell_\me)\},
\]
where $C(\mathcal G)$ denotes, of course, the space of scalar-valued functions defined on $\mathcal G$ that are continuous on each edge \textit{and also in the vertices}.

 One immediately sees that $D_{cont}$ is dense in $L^2(\mathcal G)$. Taking the restriction onto $D_{cont}$ of the sesquilinear form
 \[
 \f(u,v)=\sum\limits_{\me\in\mE}\int_0^{\ell_\me}u_\me''(x)\overline{v_\me''(x)}\,dx
 \]
 and closing it up one finds the symmetric, positive semidefinite form $(\f,D_F(\f))$: it is the form associated with the  \textit{Friedrichs extension} $A_F$ of $A_{|D_{cont}}$. Its domain $D_F(\f)$ consists of the space of all functions in $\bigoplus_{\me\in\mE} H^4(0,\ell_\me)$ 
\begin{itemize}
\item that are continuous in the vertices;
\item whose first derivatives vanish in the vertices;
\end{itemize}
self-adjointness of $A_F$ then uniquely determine the remaining transmission conditions in the vertices, which amount to further restricting to functions 
\begin{itemize}
\item whose third normal derivatives satisfy a Kirchhoff condition.
\end{itemize}
Summing up we have found the transmission conditions
\begin{equation}\label{condNic-friedr}
\begin{cases}
u_\me(\mv)=u_\mf(\mv)& \textrm{if}\ \me\cap \mf=\mv,\\
\frac{\partial u_\me}{\partial \nu}(\mv)=0& \forall\ \mv\in \mV,\\
\sum\limits_{\me\in \mE_{\mv}}	\frac{\partial^3 u_\me}{\partial \nu^3}(\mv)=0& \forall\ \mv\in \mV.
\end{cases}
\end{equation}
Observe that the second condition, which is outer first derivatives zero at each vertex, implies that no conditions are imposed on the second derivatives of the function.
In our usual formalism, these conditions correspond to the choice
\[
Y=\cmv\times \{0_{\C^{2E}}\},\quad R=0.
\]

(If $E=1$, i.e., the network consists of just one interval, then the above Friedrichs extension is  the bi-Laplacian with sliding boundary conditions. We emphasize that this construction strongly depends on our choice of $A_{|D_{cont}(A)}$ as a reference operator; had we chosen $A_{|C^\infty_c(0,\ell_\me)}$ and not bothered to enforce continuity as a means to describe the graph's connectivity, then we would have found the bi-Laplacian on a collection of disconnected edges with clamped conditions instead. Different Friedrichs extensions have actually been discussed in the literature, if the minimal operator is endowed  at a given subset of vertices with zero conditions on the boundary values of $u$~\cite[\S~8.3]{DagZua06}, or else on the boundary values of both $u$ and $u'$~\cite[\S~2]{DekNic99}.)
\end{ex}

\begin{ex}
Let us finally look for the \textit{Krein--von Neumann extension} $A_K$ of $A_{|D_{cont}(A)}$. In order to determine it, let us adapt the computations in~\cite[\S~4]{Mug15b}: we look for a self-adjoint extension with maximal form domain $\tilde{H}^2_Y(\mathcal G)$ and whose additional boundary conditions
\begin{equation*}
{\bf u}^{3,2}+\Lambda {\bf u}^{0,1}\in Y^\perp
\end{equation*}
are trimmed to make the extension as small as possible. Here $\Lambda$ is a self-adjoint operator acting on $\C^{4E}$; furthermore, $Y$ should be maximal among those subspaces that encode the continuity conditions, hence 
a natural candidate is $Y=\cmv\times \C^{2E}$.
The associated quadratic forms are then given by
\begin{equation*}\label{eq:aform}
u\mapsto \sum\limits_{\me\in\mE}\int_0^{\ell_\me}|u_\me''|^2 \ dx - 
\left(\Lambda {\bf u}^{0,1},{\bf u}^{0,1}\right).
\end{equation*}
This form is clearly non-negative if $\Lambda$ is negative semidefinite; our goal is to find the largest $\Lambda$ for which the estimate
\begin{equation}\label{eq:thissum}
\sum\limits_{\me\in\mE}\int_0^{\ell_\me}|u_\me''|^2 \ dx \ge  
\left(\Lambda {\bf u}^{0,1},{\bf u}^{0,1}\right)
\end{equation}
is still satisfied. By Hölder's inequality
\[
\int_0^{\ell_\me}|w|^2 dx\ge \frac{1}{\ell_\me} \left| \int_0^{\ell_\me} w\,dx\right|^2 \qquad \hbox{for all }w\in L^2(0,\ell_\me),
\]
and this estimate is optimal: the right-hand side in~\eqref{eq:thissum} is made as small as possible by enforcing
\[
\left(\Lambda {\bf u}^{0,1},{\bf u}^{0,1}\right)=\sum\limits_{\me\in\mE}\frac{1}{\ell_\me}\left|\int_0^{\ell_\me}u_\me''\,dx\right|^2,
\]
i.e.,
\begin{equation}\label{eq:matmultkr}
\left(\Lambda {\bf u}^{0,1},{\bf u}^{0,1}\right)=
\sum\limits_{\me\in \mE}  \frac{|u'(\me^+)- u'(\me^-)|^2}{\ell_\me },
\end{equation}
where we have denoted by $\me^+$ (resp.\ $\me^-$) the terminal (resp., initial) endpoint of $\me$.

A direct matrix multiplication shows that~\eqref{eq:matmultkr} is satisfied if $\Lambda$ has a block structure 
\[
\Lambda=\begin{pmatrix}
\lineskip=0pt
0 &\linie & 0\\ 
\noalign{\hrule}
0 &\linie & \Lambda_0
\end{pmatrix}
\]
with $\Lambda_0:=\frac12 JL^{-1}J$, where $L:=\diag(\ell_{\me_1},\ldots,\ell_{\me_E},\ell_{\me_1},\ldots,\ell_{\me_E})$ and the symmetric matrix $J$ is defined by
\[
J:=\begin{pmatrix}
\lineskip=0pt
\Id_E &\linie & \Id_E\\ 
\noalign{\hrule}
\Id_E &\linie & \Id_E
\end{pmatrix}\ .
\]
Summing up, we have found that the  Krein--von Neumann extension of $A_{|D_{cont}(A)}$ is determined by the non-local transmission conditions 
\begin{equation}\label{condNic-krein}
\begin{cases}
u_\me(\mv)=u_\mf(\mv)& \textrm{if}\ \me\cap \mf=\mv,\\
-u''_\me(0)=-u''_\me(\ell_\me)=\frac{u'_\me(\ell_\me)-u'_\me(0)}{\ell_\me}& \forall\ \me\in \mE,\\
\sum\limits_{\me\in \mE_{\mv}}	\frac{\partial^3 u_\me}{\partial \nu^3}(\mv)=0& \forall\ \mv\in \mV,
\end{cases}
\end{equation}
corresponding to the choice
\[Y=\cmv\times \C^{2E},\quad R=\Lambda\ .\]
\end{ex}

We conclude this section by explicitly presenting the consequences of our form methods on the analysis of evolution equations.
Even in the non-self-adjoint case, the sesquilinear form $\f$ enjoys good properties, which we summarize in the following. 
\begin{prop}\label{prop:gener}
Under the Assumptions~\ref{ass:standing}, the form $\f$ in~\eqref{eq:formmain}--\eqref{eq:formdomain} is densely defined and continuous.
Furthermore, it is \emph{$L^2(\mathcal G)$-elliptic}, 
i.e., there exist $\mu>0$ and $\omega\in\R$ such that 
\[\Real \f(u)+\omega||u||_{L^2(\mathcal{G})}^2\ge \mu ||u||_{D(\f)}^2\quad \hbox{for all }u\in D(\f),
\]
and of \emph{Lions type}, i.e., there exists $M>0$ such that 
\[
|\Ima \f(u)|\leq M\Vert u\Vert_{D(\f)}\Vert u\Vert_{L^2(\mathcal G)}\quad \hbox{for all }u\in D(\f).
\] 
Thus, the associated operator $-A$ generates a cosine operator function with phase space $\h^2_Y(\mathcal G)\times L^2(\mathcal G)$ and in particular a strongly continuous semigroup $(e^{-tA})_{t\ge 0}$ in $L^2(\mathcal G)$ that is analytic of angle $\frac{\pi}{2}$. Moreover, the  semigroup is self-adjoint if and only if $R$ is self-adjoint; it is real (i.e., it maps real-valued functions to real-valued functions) if and only if both the matrices $P_Y$ and $R$ are real; it is contractive if $R$ is dissipative. 
Finally, $e^{-tA}$ is of trace class (and in particular compact) for all $t>0$. 
\end{prop}
Here and in the following we denote by $P_Y$ the orthogonal projector of $\C^{4E}$ onto $Y$. We emphasize that dissipativity of $R$ is not necessary for contractivity of $(e^{-tA})_{t\ge 0}$, the Krein--von Neumann realization of the bi-Laplacian being a counterexample.
\begin{proof}
Observe that $\bigoplus_{\me\in\mE} C^\infty_c(0,\ell_\me)\subset D(\f)$: hence $D(\f)$ is dense in $L^2(\mathcal G)$. We have already observed that $\f$ is closed and \eqref{norm} holds, then $\f$ is $L^2(\mathcal G)$-elliptic.  As a consequence of the Cauchy--Schwarz inequality, of boundedness of $R$ and of Lemma \ref{lem:trace}, $\f$ is also bounded. Moreover, $\f$ is of Lions type: indeed,  $\f$ can be written as $\f(u,v)=\f_1(u,v)+\f_2(u,v)$ with
\[\f_1(u,v):=\sum\limits_{\me\in\mE}\int_0^{\ell_\me}u_\me''(x)\overline{u_\me''(x)}\,dx, \qquad \f_2(u,v):=-\left( R{\bf u}^{0,1},{\bf v}^{0,1}\right) .\] 
Now, $\f_1$ is symmetric and hence of Lions type, and $\f_2$ satisfies 
\[
|\f_2(u,u)|\leq c\Vert R\Vert_{\bound(Y)}\Vert u\Vert_{D(\f)}\Vert u\Vert_{L^2(\mathcal G)}
\]
thanks to Lemma \ref{lem:trace}. Therefore, see for example \cite[\S~5.6.6]{Are04},  $-A$ generates a cosine operator function and hence by \cite[Theorem 3.14.17]{AreBatHie01} also an analytic semigroup of angle $\frac{\pi}{2}$. 
Self-adjointness follows from symmetry of the form if $R$ is self-adjoint. The characterization of reality follows from~\cite[Prop.~2.5]{Ouh05}, reasoning as in~\cite[Prop.~4.2]{CarMug09}. If $R$ is dissipative, then $\f$ is accretive and $(e^{-tA})_{t\geq0}$ is thus contractive.
Finally, by~\cite[Satz~1]{Gra68}, the embedding of $\iota:\h^2(\mathcal G)\hookrightarrow L^2(\mathcal G)$ is of trace class, hence so is the semigroup on $L^2(\mathcal G)$ in view of the factorization
\[
e^{-t\Delta^2}_{L^2\to L^2}=\iota\circ e^{-t\Delta^2}_{\h^2\to L^2}, \qquad t>0\ .
\]
This concludes the proof.
\end{proof}

 Therefore, we obtain the following. 
 
\begin{prop}
For all $f\in L^2(\mathcal G)$  the parabolic problem associated with $-A_{Y,R}$
\begin{equation}\label{problem}\begin{cases} \frac{\partial u}{\partial t}(t,x)=-\frac{\partial ^4u}{\partial x^4}(t,x) & t\geq 0,\, x\in\mathcal G,\\ \begin{pmatrix} u(t,0)\\u(t,\ell)\\-u'(t,0)\\ u'(t,\ell)\end{pmatrix}\in Y & t\geq0,\\
\begin{pmatrix} -u'''(t,0)\\u'''(t,\ell)\\-u''(t,0)\\-u''(t,\ell)\end{pmatrix}+R\begin{pmatrix} u(t,0)\\u(t,\ell)\\-u'(t,0)\\u'(t,\ell)\end{pmatrix}\in Y^\perp & t\geq0,
\\
 u(0,x)=f(x), & x\in\mathcal G\end{cases} \end{equation}
  admits  a unique solution in $L^2(\mathcal G)$: this is given by $u(t,\cdot)=e^{-tA_{Y,R}}f(\cdot)$, $t\ge 0$. 
\end{prop}
A similar assertion holds for the corresponding second order problem with initial conditions
\[
u(0,x)=f(x)\quad\hbox{and}\quad \frac{\partial u}{\partial t}(0,x)=g(x),\qquad x\in\mathcal G,
\]
for all $f\in D(\f)$ and all $g\in L^2(\mathcal G)$.

\begin{rem}\label{rem:federica} 
The first author studied in \cite{GreMil16} the bi-Laplacian operator perturbed by singular potentials in higher dimensions. In particular, she discussed
\[A=-\Delta^2+\frac{c}{|x|^4}\]
in dimension $N\geq5$ and where $c<\left(\frac{N(N-4)}{4}\right)^2$ and proved the  generation of a bounded
holomorphic strongly continuous semigroup on $L^p(\R^N)$ for 
$p$ in an open interval around 2.
The constraint on the constant $c$ is due to the best constant in the Rellich inequality which is a strong tool to prove generation by form methods. Since the Rellich inequality is also valid on the half-line, see \cite{Rel54}, generation results for $A=\frac{d^4}{dx^4}-\frac{c}{x^4}$ could be  obtained straightforwardly at least on $L^2(\mathcal{G})$ on a suitable infinite graph, for example a star consisting of $k$ semi-infinite intervals. 
\end{rem}

\section{Contractivity issues}\label{sec:contra}

In the remainder of this paper, we are going to study further properties of semigroups governing on fourth-order parabolic equations, focusing on similarities with, and differences from, well-known features of heat semigroups generated by usual Laplacians.

\subsection{Ultracontractivity}\label{ultrac}

Because $e^{-tA}$ is for all $t>0$ a trace class  operator and therefore a Hilbert--Schmidt operator, it has an integral kernel of class $L^2(\mathcal G\times \mathcal G)$, i.e., there exists a family $(p(t,\cdot,\cdot))_{t> 0}$ of $L^2(\mathcal G\times \mathcal G)$-kernels such that
the semigroup operators $e^{-tA}$ satisfy
\[e^{-tA}f(x)=\int_\mathcal G p(t,x,y) f(y)\,dy\qquad\hbox{for all }t>0\hbox{ and }f\in L^2(\mathcal G)\hbox{ and a.e. } x\in\mathcal G. \]
Indeed, more is true: 
By the Kantorovitch--Vulikh Theorem the space of bounded linear operators from $L^1(X)$ to $L^\infty(X)$ is isometrically isomorphic to the space of $L^\infty(X\times X)$-kernels whenever $X$ is a $\sigma$-finite measure space (see~\cite{MugNit11} for several generalizations).
Our analytic semigroup boundedly maps $L^2(\mathcal G)$ to $\h^1(\mathcal G)$, hence in particular each operator $e^{-tA_{Y,R}}$  boundedly maps $L^2(\mathcal G)$ into $L^\infty(\mathcal G)$ and so does $e^{-tA_{Y,R^*}}$; by duality both $e^{-tA_{Y,R^*}}$ and $e^{-tA_{Y,R}}$  boundedly map  $L^1(\mathcal G)$ into $L^2(\mathcal G)$. Hence $L^1(\mathcal G)$ to $L^\infty(\mathcal G)$ and by the Kantorovitch--Vulikh Theorem $p(t,\cdot,\cdot)\in L^\infty(\mathcal G\times \mathcal G)$ (indeed, if the form domain $D(\f)$ encodes continuity condition in the vertices and is hence embedded in $C(\mathcal G)$, then $p(t,\cdot,\cdot)\in C(\mathcal G\times \mathcal G)$  by the recent result~\cite[Thm.~2.1]{AreEls17}). It is thus natural to wonder whether $\|T(t)\|_{1\to \infty}$ (or equivalently $\|p(t,\cdot,\cdot)\|_\infty$) can be nicely estimated.

Given a metric measure space $X$, a semigroup $(T(t))_{t\ge 0}$ on $L^2(X)$ is called \textit{ultracontractive} if there exist $q\in (2,\infty]$, $\beta>0$, and $c>0$ such that
\[
\|T(t)\|_{2\to q}\le ct^{-\beta},\qquad t\in (0,1].
\]

We emphasize that ultracontractivity of semigroups generated by second-order elliptic operators is usually showed taking advantage of their $L^\infty$-contractivity, which unfortunately fails in the fourth-order context. Indeed, it seems that ultracontractivity cannot be deduced from standard Sobolev embeddings~\cite[Thm.~6.4]{Ouh05}, since if $u\in D(\f)$ then  $(1\wedge|u|)\,\textrm{sign}\, u\notin \h^2(\mathcal G)$, as its derivative is not even a continuous function; accordingly, Ouhabaz' criterion \cite[Theorem 2.13]{Ouh05} is not satisfied -- regardless of the choice of $Y$ and $R$ -- hence $(e^{-tA})_{t\ge 0}$ is never contractive with respect to the $\infty$-norm. 

It is known that under the additional assumptions that
\[
\sup_{t\in (0,1]}\|T(t)\|_1<\infty\quad\hbox{and}\quad \sup_{t\in (0,1]}\|T(t)\|_\infty<\infty,
\]
ultracontractivity implies the existence of $M>0$ and $\omega\in \mathbb R$ such that
\begin{equation*}\label{eq:defultra}
\|T(t)\|_{1\to \infty}\le Mt^{-2\beta}e^{\omega t},\qquad t>0.
\end{equation*}
Of course, this semigroup is not strongly continuous in $L^\infty(\mathcal G)$ and hence it needs not be exponentially bounded. Failing $L^\infty$-contractivity of $(e^{-tA})_{t\ge 0}$ it is, therefore, a priori not clear whether ultracontractivity would imply the estimate~\eqref{eq:defultra}. However, the following holds.

\begin{prop}\label{prop:ultracontr} 
Under our standing assumption \ref{ass:standing}, the semigroup $(e^{-tA})_{t\geq0}$ 
has an integral kernel of class $L^\infty(\mathcal G\times \mathcal G)$.
If we additionally assume $R$ to be dissipative, then there exist $\varepsilon,M>0$ such that
\begin{equation}\label{1inf}
\Vert e^{-tA}\Vert_{1\to\infty}\leq M t^{-\frac{1}{4}} e^{\varepsilon t}\quad \hbox{for all }t> 0
\end{equation}
holds. In particular, $(e^{-tA})_{t\geq0}$ is ultracontractive.
\end{prop}

\begin{proof}
Let us first recall that $R$ dissipative implies $(e^{-tA})_{t\geq0}$ contractive in $L^2(\mathcal{G})$. 
Consider $u\in\h^2(\mathcal G)$: by an application of the Gagliardo--Nirenberg inequality one obtains
\begin{equation*}
\Vert u\Vert_{L^\infty(\mathcal G)}\leq c_1\Vert u''\Vert_{L^2(\mathcal G)}^{1/4}\Vert u\Vert_{L^2(\mathcal G)}^{3/4}+c_2\Vert u\Vert_{L^2(\mathcal G)}.
\end{equation*}
Then, since $R$ is dissipative
\begin{align}\label{interp}
\Vert u\Vert_{L^\infty(\mathcal{G})}&\leq c_1\left(\f(u)+\left(R{\bf u}^{0,1},{\bf u}^{0,1}\right) \right)^{1/8}\Vert u\Vert_{L^2(\mathcal G)}^{3/4}+c_2\Vert u\Vert_{L^2(\mathcal G)}\nonumber\\
&\leq c_1 \f(u)^{1/8}\Vert u\Vert_{L^2(\mathcal G)}^{3/4}+c_2\Vert u\Vert_{L^2(\mathcal G)}.
\end{align}
Let $f_t=e^{-tA}f$  then by \eqref{interp}, contractivity and analyticity on $L^2$ and Cauchy--Schwarz inequality  it follows that for all $t>0$
\begin{align}\label{ultracontratt}
\Vert e^{-tA}f\Vert_{L^\infty(\mathcal G)}&
\leq c_1 (Ae^{-tA}f,e^{-tA}f)^{1/8}\Vert e^{-tA}f\Vert_{L^2(\mathcal G)}^{3/4}+c_2\Vert e^{-tA}f\Vert_{L^2(\mathcal G)}\nonumber\\
&\leq c_1 \Vert Ae^{-tA}f\Vert_{L^2(\mathcal G)}^{1/8}\Vert e^{-tA}f\Vert_{L^2(\mathcal G)}^{1/8}\Vert f\Vert_{L^2(\mathcal G)}^{3/4}+c_2\Vert f\Vert_{L^2(\mathcal G)}\nonumber\\
&\leq c(t^{-1/8}+1)\Vert f\Vert_{L^2(\mathcal G)}\nonumber\\
&\leq c t^{-1/8}e^{\varepsilon t}\Vert f\Vert_{L^2(\mathcal G)}.
\end{align}
Therefore, $(e^{-tA})_{t\geq0}$ is ultracontractive. 
 Now, apply \eqref{ultracontratt} to the semigroup generated by $A^*$. By duality this yields the bound \eqref{ultracontratt} to hold for the $L^1-L^2$ norm as well; and then
\begin{equation*}\label{re}
 \Vert e^{-tA}f\Vert_{L^\infty(\mathcal G)}\leq M t^{-1/4}e^{\varepsilon t}\Vert f\Vert_{L^1(\mathcal G)},
\end{equation*}
as we wanted to prove.
\end{proof}
Observe that, from the computations in \eqref{ultracontratt}, one also finds
\[
\Vert e^{-tA}f\Vert_{L^\infty(\mathcal G)}\le c(t^{-1/8}+1)^2\Vert f\Vert_{L^1(\mathcal G)}\approx 
C\Vert f\Vert_{L^1(\mathcal G)}\qquad \hbox{as }t\to \infty.
\]



Let us remark that even if it is not $L^\infty$-contractive, the semigroup does nevertheless extrapolate to further $L^p(\mathcal G)$-spaces.

\begin{theorem}\label{thm:extrap}
 Let $Y$ be a subspace of $\C^{4E}$ and $R\in\mathcal{L}(Y)$ be dissipative.
The semigroup $(e^{-tA})_{t\geq0}$ on $L^2(\mathcal G)$ extrapolates to a consistent family of semigroups on $L^p(\mathcal G)$ for every $1\leq p\leq\infty$: they satisfy
\[
\Vert e^{-tA}\Vert_{L^p(\mathcal G)}\leq c t^{-\frac{|2-p|}{4p}}e^{\varepsilon t\frac{|2-p|}{p}},\quad t> 0,\]
for every $1\leq p<\infty$ as well as
\[\Vert e^{-tA}\Vert_{L^\infty(\mathcal G)}\leq c t^{-\frac{1}{4}}e^{\varepsilon t},\quad t> 0.\] 
\end{theorem}
\begin{proof}
We first observe that since $L^\infty(\mathcal{G})\subset L^1(\mathcal{G})$ and  \eqref{1inf} holds, there exists two positive constants such that for every $t>0$
\[\Vert e^{-tA}f\Vert_1\leq c\Vert e^{-tA}f\Vert_\infty\leq ct^{-1/4}e^{\varepsilon t}\Vert f\Vert_1.\]
By interpolating with the $L^2$-contractivity and duality we find the claimed $L^p-L^p$ bound. 
\end{proof}


\subsection{$L^p$-contractivity of the semigroup generated by the bi-Laplacian on an interval}

Let us now focus on contractivity issues for the semigroup $(e^{-tA})_{t\ge 0}$. We already know that it is contractive with respect to the norm of $L^2(\mathcal{G})$ when $R$ is dissipative, which is the case for the realizations of $A$ considered in the following. 
We are going to prove that contractivity extends to the extrapolated semigroups on $L^p(\mathcal G)$ for an open interval around $2$. (Observe that, by Riesz--Thorin, the set of those $p$ for which $(e^{-tA})_{t\ge 0}$ is $L^p$-contractive has to be connected. On the other hand, we do already know that the semigroup is certainly \textit{not} $L^\infty$-contractive.) 
To this purpose we make use of a characterization of $L^p$-contractivity of semigroups associated with a form $\f$ due to Nittka \cite[Thm.~4.1]{Nit12}: It states that a given semigroup on a $L^2$-space is $L^p$-contractive if and only if the conditions  for $p\geq2$
\begin{equation}\label{eq:carcontr-1}
P_{B^p}(D(\f))\subset D(\f)
\end{equation}
and
\begin{equation}\label{eq:carcontr-2}
\Real \f(u,|u|^{p-2}u)\geq0\ \textrm{for every}\ u\in D(\f)\ \textrm{satisfying}\ |u|^{p-2}u\in D(\f),
\end{equation}
are satisfied, where $P_{B^p}$ denotes the projection of $L^2$ onto the unit ball of $L^p$. The most interesting feature of this characterization is that $P_{B^p}$ cannot be written down explicitly, still an implicit description is useful enough, in most applications: if $u\in L^2$ and $f=P_{B^p}u$, then 
\begin{equation}\label{projection}
u=f+(u-f,f) |f|^{p-2}f.
\end{equation}
It seems that Nittka's elegant characterization has not found many applications so far. As a warm-up, let us consider a simple case which seems to be of independent interest.

Let us consider the bi-Laplacian acting on functions in $L^p(0,1)$ under rather general boundary conditions: we are actually only going to impose Neumann boundary conditions and do not specify further conditions.

\begin{lemma}\label{lem:p-contr}
Consider $\mathcal{G}$ consisting of a single interval, i.e., $\mathcal G=(0,1)$.
The semigroup associated with the quadratic form
\[
\f(u):=\int_0^1 |u''(x)|^2 dx
\]
is $L^p$-contractive for all $p\in [3/2,3]$ whenever its form domain is 
\begin{equation*}
\left\{u\in H^2(0,1):\begin{pmatrix}u(0)\\ u(1) \end{pmatrix}\in Y_1, \begin{pmatrix}u'(0)\\ u'(1) \end{pmatrix}=\begin{pmatrix}0\\ 0 \end{pmatrix} \right\}.
\end{equation*}
for $Y_1=\{(0,0)\}$ or $Y_1=\langle(1,1)\rangle$.
\end{lemma}

We remark that in the language of networks, both choices of $Y_1$ correspond to Friedrichs realization of the bi-Laplacian: either on a graph consisting of a single edge or on a loop, respectively.

\begin{proof}
It will suffice to prove the assertion for $p=3$, since the case $p=3/2$ follows by duality and all remaining cases by interpolation. Furthermore, we can assume without loss of generality that our ambient spaces $L^2(0,1),L^p(0,1)$ are real: indeed, it follows from the results in \cite{GM94} (see also \cite{HK05}) that the norm of $e^{-tA}$ as an operator on $L^p(0,1;\R)$ coincides with  the norm of $e^{-tA}$ as an operator on $L^p(0,1;\C)$.
To begin with, let us check condition~\eqref{eq:carcontr-1}, i.e., let us prove that the projection onto $B^p$ of a function $u\in H^2(0,1)$ is still twice weakly differentiable with $P_{B^p} u,P_{B^p} u',P_{B^p} u''\in L^2(0,1)$; and that the boundary conditions are still satisfied. 
Let $u\in H^2(0,1)\setminus B^p$ and denote by $f=P_{B^p}u$ its projection which satisfies $\Vert f\Vert_p=1$ and $f+t|f|^{p-2}f=u$ with $t:=(u-f,f)$. By ~\cite[Thm.~4.3]{Nit12}, we know that $f=\varphi_t^{-1}\circ u\in H^1(0,1)$
 where $\varphi_t(x)=x+t|x|^{p-2}x$, and $\varphi_t^{-1}\in W^{1,\infty}(0,1)$. In order to obtain $f\in H^2(0,1)$ we observe that \[D(D(\varphi_t^{-1}\circ u))=(\varphi_t^{-1})''(u)\cdot(u')^2+(\varphi_t^{-1})'(u)\cdot u''\]
which belongs to $L^2(0,1)$ since one can write $(\varphi_t^{-1})''(a)=-\frac{\varphi''_t(\varphi_t^{-1}(a))\cdot(\varphi_t^{-1})'(a)}{(\varphi'_t(\varphi_t^{-1}(a)))^2}$ and $\varphi''_t\in L^\infty(0,1)$ for $p=3$. Taking into account \eqref{projection}, if $u$ satisfies $\begin{pmatrix}u(0)\\ u(1) \end{pmatrix}\in Y_1$, for some $Y_1$ as in the statement, and $\begin{pmatrix}u'(0)\\ u'(1) \end{pmatrix}=\begin{pmatrix}0\\ 0 \end{pmatrix}$, then $f$ does too in view of $f'=D(\varphi_t^{-1}\circ u)=(\varphi_t^{-1})'(u)\cdot u'$.
Furthermore, we see that 
\begin{align*}
\f(u,|u|^{p-2}u)&=(p-1)\int_0^1 |u(x)|^{p-2}(u''(x))^2\,dx+(p-1)(p-2)\int_0^1|u(x)|^{p-3}\textrm{sign}\,u (u'(x))^2u''(x)\,dx\\
&=(p-1)\int_0^1 |u(x)|^{p-2}(u''(x))^2\,dx-\frac{(p-1)(p-2)(p-3)}{3}\int_0^1|u(x)|^{p-6}(u(x))^2(u'(x))^4\,dx\\
&\quad +\frac{(p-1)(p-2)}{3}\left[|u(x)|^{p-3}\textrm{sign}\, u (u'(x))^3\right]_0^1.
\end{align*}
Taking $p=3$ and the boundary conditions into account one has 
\begin{align*}
\f(u,|u|u)=2\int_0^1 |u(x)|(u''(x))^2\,dx\geq0 , 
\end{align*}
as we wanted to show.
\end{proof}

It is easy to see that the above proof also carries over verbatim to different boundary conditions on the trace of $u$, especially $Y_1=\C^2$ (no condition) and $Y_1=\langle(1,-1)\rangle$ (anti-periodic boundary conditions).

\subsection{$L^p$-contractivity of the semigroup generated by the bi-Laplacian on a network}\label{sec:lp-contr}

Let us extend Lemma~\ref{lem:p-contr} to the case of a general network; this shows that a counterpart of Proposition~\ref{prop:ellp-discr} holds in the continuous case, too.

\begin{prop}
The semigroup generated by the Friedrichs realization of the bi-Laplacian (see Example~\ref{ex-nickur})
is $L^p(\mathcal G)$-contractive for all $p\in [3/2,3]$.
\end{prop}

\begin{proof}
We are going again to check Nittka's conditions \eqref{eq:carcontr-1} and \eqref{eq:carcontr-2} for $p=3$.

To begin with, we observe that $P_{B^p}\left(\bigoplus_{\me\in \mE} H^2(0,\ell_\me)\right)\subset \bigoplus_{\me\in \mE} H^2(0,\ell_\me)$ can be proved precisely as in Lemma~\ref{lem:p-contr}. It remains to check that $P_{B^p}$ respects the boundary condition, or more generally that
\[
P_{B^p}{\bf y}\in Y\quad\hbox{whenever}\quad {\bf y}\in Y.
\]
By~\cite[Lemma~2.3]{ManVogVoi05}, this is equivalent to asking that $P_Y B^p\subset B^p$, where $P_Y$ denotes the orthogonal projector of $\C^{4\mE}$ onto $Y$: but in the case of the Friedrichs realization, \[
P_Y=\begin{pmatrix}
\lineskip=0pt
P_{\langle c_\mV \rangle} &\linie & 0\\ 
\noalign{\hrule}
0 &\linie & 0
\end{pmatrix}
\]
and it is well-known that $P_{\cmv}$ leaves $B^p$ invariant (e.g., as a consequence of the Markovian property of the semigroup generated by the Laplacian with continuity and Kirchhoff vertex conditions on $\mathcal G$; or else, more directly, because $P_{\cmv}$ is easily seen to leave invariant both $B^2$ and $B^\infty$).
This takes care of~\eqref{eq:carcontr-1}.

Let us now turn to~\eqref{eq:carcontr-2}: we reason as above and compute
\begin{align*}
\f(u,|u|^{p-2}u)&=(p-1)\sum\limits_{\me\in\mE}\int_0^{\ell_\me} |u_\me(x)|^{p-2}(u_\me''(x))^2\,dx\\&\quad-\frac{(p-1)(p-2)(p-3)}{3}\sum\limits_{\me\in\mE}\int_0^{\ell_\me}|u_\me(x)|^{p-6}(u_\me(x))^2(u_\me'(x))^4\,dx\\&\quad
 +\frac{(p-1)(p-2)}{3}\sum\limits_{\me\in\mE}\left[|u_\me(x)|^{p-3}\sign\, u_\me (u_\me')^3\right]_0^1.
\end{align*}
For $p=3$ and owing to the Neumann boundary conditions satisfied by the functions in the domain of the Friedrichs realization we find
\[\f(u,|u|u)=2\sum\limits_{\me\in\mE}\int_0^{\ell_\me}|u_\me(x)|(u_\me''(x))^2\,dx
\]
which is clearly non-negative.
\end{proof}

We already know that these semigroups are never $L^\infty$-contractive (in the continuous case), resp.\ if and only if $\mG$ is complete (in the discrete case).  We leave it as an open problem to find the precise $p_0$ at which  transition from $L^p$-contractivity to $L^p$-\textit{non}-contractivity occurs.

\section{Positivity and Markovianity issues}\label{sec:event-pos}

This section is devoted to the investigation of whether the non-positive semigroup generated by bi-Laplacians on graphs and networks can actually become positive for large time, a phenomenon known as \textit{eventual positivity}. We refer the reader to Section~\ref{sec:app} for a reminder of this theory, including basic definitions and the main results we are going to exploit in the following.

\subsection{Eventual positivity}

Recall that by Proposition~\ref{prop-complete} the semigroup generated by the discrete bi-Laplacian is not positive unless the graph is complete. In the non-complete case, this information can be complemented as follows.

\begin{prop}\label{prop:gluthesis}
Let $\mG$ be a finite, connected graph. The semigroup $(e^{-t\mathcal L^2})_{t\ge 0}$ is eventually irreducible.
\end{prop}

\begin{proof}
Apply Corollary~\ref{cor:indivunif}.(2).
\end{proof}

\begin{rem}\label{rem:eventposspec}
1) The spectral theorem suggests that the smallest transition time $t_{\max}$ from positivity to non-positivity is related to the lowest non-zero eigenvalue. Numerical simulations support the conjecture that such $t_{\max}$ is attained when applying the semigroup to characteristic functions of the vertices of minimal degree. In the case of the path graph on 3 vertices, taking $(1,0,0)$ as initial condition leads to the solution
\[
u(t)=\begin{pmatrix}
\frac{1}{3}+\frac{1}{6}e^{-9t}+\frac{1}{2}e^{-t}\\
\frac{1}{3}-\frac{1}{3}e^{-9t}\\
\frac{1}{3}+\frac{1}{6}e^{-9t}-\frac{1}{2}e^{-t}
\end{pmatrix},\quad t\ge 0.
\]

2) It can be shown by a Galerkin scheme similar to that in~\cite{Mug13} that the solution of the parabolic equation associated with $-\mathcal L^2$ on an infinite graph $\mG$ can be approximated by solving the analogous parabolic equation on a sequence $(\mG_n)_{n\in \mathbb N}$ of graphs exhausting $\mG$: while each semigroup on $\mG_n$ is eventually irreducible, the transition time may blow up as $n\to \infty$; we leave this as an open problem.
\end{rem}

As already mentioned in the introduction, the parabolic problem associated with the (differential) bi-Laplacian is not governed by a positive semigroup. Indeed, the following holds.

\begin{lemma}
Let $L$ be a self-adjoint, uniformly elliptic operator of order $2m$ on $L^2(\Omega)$ with $\Omega\subset\R$. Then a necessary condition for the semigroup $(e^{-tL})_{t\ge 0}$ generated by $-L$ to be positive is that $m=1$. 
\end{lemma}

\begin{proof}
Let $\f$ be the quadratic form associated with $L$. Because $L$ is self-adjoint, it satisfies the square root property and hence the form domain is the Sobolev space $H^m(\Omega)$. Now, by the Beurling--Deny criteria a necessary condition for $(e^{-tL})_{t\ge 0}$ to be positive is that $u^+\in H^m(\Omega)$ whenever $u\in H^m(\Omega)$. This is the case if and only if $m=1$. Indeed, by \cite[Proposition 4.4]{Ouh05} if $u\in H^1(\Omega)$ then $u^+\in H^1(\Omega)$. On the other hand, it suffices to consider a linear function which changes sign in $\Omega$ to see that it belongs to $H^k(\Omega)$ for every $k\in\N$, but $u^+\notin H^2(\Omega)$ and then in any $H^k(\Omega)$ for $k\geq2$.
\end{proof}

\begin{ex}
With the terminology presented in the introduction, $-\frac{d^4}{dx^4}$ generates a semigroup on $L^2(0,1)$ that in view of Corollary~\ref{cor:indivunif} is
\begin{itemize}
\item eventually irreducible provided hinged, clamped, or sliding conditions are imposed;
\item not even individually asymptotically positive provided free boundary conditions are imposed.
\end{itemize}
\end{ex}

In the case of general networks we are going to apply Corollary~\ref{cor:indivunif}.(2) in order to inquire eventual irreducibility of the semigroup generated by the bi-Laplacian. 
To begin with, we observe that every realization of the bi-Laplacian $A$ on a network of finite total length has compact resolvent (in fact, even trace-class resolvent by Proposition~\ref{prop:gener}), hence it has purely point spectrum. If $A$ is self-adjoint and hence its vertex conditions can be written as in~\eqref{condR}, then the spectrum of $-A$ is, of course, real; it is also nonnegative -- say, $0\le \lambda_1\le \lambda_2\le\ldots$ -- if additionally $R$ is dissipative or $R=\Lambda$ (as in the Krein--von Neumann extension).
The constant function ${\bf 1}$ is contained in $\ker A$ provided \eqref{cont}
holds; the associated eigenprojector is, of course, positive. Whether  the null space has  multiplicity $>1$ influences the behavior of the semigroup generated by the bi-Laplacian in several ways beyond the mere structure of stationary solutions of~\eqref{problem}. In order to establish results on the eventual irreducibility of the semigroup generated by different realizations of the bi-Laplacian we have then to study the multiplicity of the eigenvalue $0$ depending on the transmission conditions in the vertices. In most situations, this can be done by a straightforward computation.
Observe that when the realization of the bi-Laplacian can be expressed with  vertex conditions of type~\eqref{condR}  but  with $R=0$  each element of $\ker A$ is a harmonic, hence edgewise affine function, i.e., $u_\me(x)=a_\me x+b_\me$ for every $\me\in\mE$.

In the following we assume as usual $\mathcal G$ to be a finite, connected network on $V$ vertices and $E$ edges.

We are finally able to obtain the following results on eventual positivity depending on the the transmission conditions in the vertices and/or the network topology.

\begin{prop}\label{prop:event-friedkrein}
The semigroup generated by minus the Friedrichs extension $A_F$ of the bi-Laplacian $(A,D_{cont}(A))$ is eventually irreducible, and so is the semigroup generated by $-A$ with conditions~\eqref{condNic1}. The semigroup generated by minus the Krein--von Neumann extension $A_K$ of the bi-Laplacian $(A,D_{cont}(A))$ is not even individually asymptotically positive, nor is the semigroup generated by $-A$ with conditions \eqref{condNic3}.
\end{prop}

\begin{proof}
We are going to deduce these assertions from Corollary~\ref{cor:indivunif}.
The first assertion follows by observing that the null space of the Friedrichs extension only contains functions that are constant over the whole network.

On the other hand, the null space of the Krein extension consists of functions that are continuous in the vertices, edgewise polynomials of degree at most 3, and in fact of degree at most 2 in view of the non-local boundary condition in~\eqref{condNic-krein}: this yields $\sum_{\mv\in\mV}(\deg(\mv)-1)=2E-V$ and $E$ constraints, respectively. Hence, the space of bi-harmonic functions satisfying conditions~\eqref{condNic-krein} is a space of dimension $4E-(2E-V)-E$.
Hence the spectral bound $0$ is an eigenvalue of multiplicity $E+V$ of the Krein--von Neumann extension. 

In the case of conditions~\eqref{condNic3}, only continuity on the trace but no conditions on the first derivative are imposed in the vertices, hence all polynomials of degree $\le 1$ whose coefficients realize the continuity condition lie in the null space. (We could also invoke \cite[Theorem 3.2]{DekNic00} for this remark, since it is proved that in this case $\dim\ker A=V$.) 
\end{proof}

If we specialize the graph to the case of cycles we can obtain the following.

\begin{prop}\label{prop:oddc}
Let $\mathcal G$ be a cycle. Then the semigroup generated by the operator $- A$ with conditions~\eqref{condNic2} in all vertices is
\begin{itemize}
\item  eventually irreducible if $\mathcal G$ has an odd number of edges;
\item not individually asymptotically positive if $\mathcal G$ has an even number of edges.
\end{itemize}  
\end{prop}
\begin{proof}
Suppose $\mathcal{G}$ has $k$ edges and consider a function $u$ in the null space of $A$. Observe that it satisfies $\f(u,u)=0$: thus it is harmonic, hence affine.

Let us first assume $k$ to be odd. Then, imposing continuity condition on the normal derivatives of $u$ in all vertices implies that $u$ is edgewise constant along the cycle. Furthermore, the first vertex condition in~\eqref{condNic2} implies that the function $u$ is exactly the same constant on the graph. Then, the assertion follows from Corollary~\ref{cor:indivunif}. 

Let now $k$ be even. Let us first consider the case when all edges  have length 1. Also, we can assume the edges to be oriented in such a way that each vertex is incident with either two incoming or two outgoing edges.
Then the mutually orthogonal functions
\[
\phi:=\begin{pmatrix}
1\\ \vdots \\ 1
\end{pmatrix},\quad 
\psi:=\begin{pmatrix}
x-\frac12\\ \vdots \\ x-\frac12
\end{pmatrix},
\]
both lie in the null space of the bi-Laplacian  with conditions~\eqref{condNic2}; accordingly, by Corollary~\ref{cor:indivunif} the associated semigroup is not individually asymptotically positive. The general case follows by rescaling the above function $\psi$.
\end{proof}


\subsection{Eventual $L^\infty$-contractivity}\label{sec:eventinfty}

We already know that neither the semigroup generated by minus the discrete bi-Laplacian $-\mathcal L^2$ nor by minus the fourth derivative on an interval (regardless of the boundary conditions) are $L^\infty$-contractive, i.e., they do not necessarily map initial data in $B^\infty_1(0)$ to solutions that lie in $B^\infty_1(0)$ for \textit{all} $t>0$.

Let us consider discrete Laplacian $\mathcal L$ on a finite graph $\mG$: by Proposition~\ref{prop:gluthesis} the semigroup $(e^{-t\mathcal L^2})_{t\ge 0}$ is eventually irreducible.
We already know that $\bf 1$ spans the null space of $\mathcal L^2$ and conclude the following.

\begin{prop}\label{prop:eventlinftydiscr}
The semigroup $(e^{-t\mathcal L^2})_{t\ge 0}$ is eventually $\ell^\infty$-contractive, hence eventually sub-Markov\-ian.
\end{prop}
\begin{rem}
An interesting consequence of Proposition~\ref{prop:eventlinftydiscr} is that $(e^{-t\mathcal L^2})_{t\ge 0}$ is uniformly bounded on $\ell^p(\mV)$ for all $p\in [1,\infty]$: indeed, let us denote by $t_0>0$ a time such that $e^{-t\mathcal L^2}$ is $\ell^\infty$-contractive and set $M:=e^{t_0\|\mathcal L\|_\infty^2}$  an upper bound for $\sup_{t\in [0,t_0]}\|e^{-t\mathcal L^2}\|_\infty$. Combining these estimates we obtain that $\|e^{-t\mathcal L^2}\|_\infty\le M$ for all $t\ge 0$, hence by duality and Riesz--Thorin interpolation $\|e^{-t\mathcal L^2}\|_p\le M$ for all $p\in [1,\infty]$.
\end{rem}

\begin{ex}
Let us consider the bi-Laplacian $A$ on $L^2(0,1)$. We already know that the realization of the bi-Laplacian with free boundary conditions generates a semigroup that is not even individually asymptotically positive, hence certainly not eventually sub-Markovian. Also under hinged, clamped, or sliding boundary conditions the form domain of $A$ is a subspace of $H^4(0,1)$ and hence the projector onto $B^\infty_1(0)$ does not leave it invariant. Accordingly, 
by the Beurling-Deny criteria the semigroup generated by $-A$ cannot be $L^\infty$-contractive under any of these boundary conditions.
However, such a semigroup is eventually $L^\infty$-contractive -- and hence eventually sub-Markovian -- under sliding boundary conditions (and only under those), since the null space of $A$ is one-dimensional, as it only contains constant functions, and the associated eigenprojector is of course a positive mapping.
\end{ex}

More generally, a direct application of Proposition~\ref{prop:glumug} yields the following.

\begin{prop}
The semigroup generated by minus the Friedrichs extension $A_F$ (i.e., the realization with conditions~\eqref{condNic-friedr}) of the bi-Laplacian $(A,D_{cont}(A))$ is eventually sub-Markovian, and so is the semigroup generated by $-A$ with conditions~\eqref{condNic1}.
\end{prop}

The same holds for the semigroup discussed in Proposition~\ref{prop:oddc}.

\subsection{Spectral issues}\label{sec:spectrum}

Let us briefly discuss some features of the bi-Laplacian spectrum.
As already mentioned, the bi-Laplacian with vertex conditions~\eqref{condNic1} is the square of the Laplacian $\Delta^{CK}$ with continuity and Kirchhoff vertex conditions, a self-adjoint, negative semi-definite operator; accordingly, $(\Delta^{CK})^2$ has pure point spectrum given by $\big\{\mu^2:\mu\hbox{ is an eigenvalue of }\Delta^{CK}\big\}$ and all spectral information available for $\Delta^{CK}$ (see e.g.~\cite{Bel85,KenKurMal16}) can be readily extended to $(\Delta^{CK})^2$.

Accordingly, under our standing assumption that $\mathcal G$ is connected $0$ is a simple eigenvalue of $\big\{\mu^2:\mu\in \sigma(\Delta^{CK})\big\}$ and the associated eigenspace is spanned by the constant function $\bf 1$, i.e., it is the space of all constant functions on $\mathcal G$: it follows that for all initial data $f\in L^2(\mathcal G)$ the solution $u$ of
\begin{equation}\label{problem-1}
\begin{cases} \frac{\partial u}{\partial t}(t,x)=-\frac{\partial ^4u}{\partial x^4}(t,x) & t\geq 0,\, x\in\mathcal G,\\ 
u_\me(t,\mv)=u_\mf(t,\mv)& \textrm{if}\ \me\cap \mf=\mv,\ t\ge 0,\\
\sum\limits_{\me\in \mE_{\mv}}\frac{\partial u_\me}{\partial \nu}(t,\mv)=0& \forall\ \mv\in \mV,\ t\ge 0,\\
\frac{\partial^2 u_\me}{\partial x^2}(t,\mv)=\frac{\partial^2 u_\mf}{\partial x^2}(t,\mv)& \textrm{if}\ \me\cap \mf=\mv,\ t\ge 0,\\
\sum\limits_{\me\in \mE_{\mv}}	\frac{\partial^3 u_\me}{\partial \nu^3}(t,\mv)=0& \forall\ \mv\in \mV,\ t\ge 0,\\
 u(0,x)=f(x), & x\in\mathcal G
 \end{cases} 
 \end{equation}
converges to the projection of $f$ onto $\langle {\mathbf 1}\rangle$, i.e., to the mean value of $f$.

Furthermore, all networks that are extremal with respect to the spectral gap of $\Delta^{CK}$ are also extremal with respect to the spectral gap of $(\Delta^{CK})^2$.
In view of the spectral theorem, estimates on eigenvalues of $\Delta^{CK}$ also influence the long-time behavior of the semigroup generated by $-(\Delta^{CK})^2$, since the orthogonal projector onto the eigenspace $\langle \bf 1\rangle$ is the uniform limit of $(e^{-t(\Delta^{CK})^2})_{t\ge 0}$ as $t\to\infty$, with convergence rate $e^{-\mu_2^2 t}$, where $\mu_2$ is the second lowest eigenvalue of $\Delta^{CK}$. Accordingly, we can e.g.\ deduce from~\cite[Théo.~3.1]{Nic87} 
that among all networks of given total length $L$ the slowest convergence to equilibrium for the solution of~\eqref{problem-1}
 is attained on an interval of length $L$. Similarly, given a network on $E\ge 3$ edges, it follows from~\cite[Thm.~4.2]{KenKurMal16} that convergence to equilibrium is fastest on equilateral pumpkin or flower graphs. 
If we try to extend this result to different realizations of the bi-Laplacian, we see that not all surgery principles in~\cite{BerKenKur19} -- the most basic tool in spectral analysis of Laplacians on networks -- carry over naively to bi-Laplacians $A_{Y,R}$ under general transmission conditions in the vertices. We will focus on this issue in a forthcoming paper. 

It \textit{is} however true that the operation of joining two or more different vertices of the metric graph leads to larger eigenvalues provided it makes the form domain smaller: this is e.g.\ the case if conditions~\eqref{condNic2}, \eqref{condNic3}, \eqref{condNic-friedr}, and~\eqref{condNic-krein} (but not~\eqref{condNic1}) are imposed in each vertex: in these cases, the bi-Laplacian has largest $k$-th eigenvalue, for some or equivalently for all $k\in\mathbb N$ -- and in particular the rate of convergence of the semigroup  towards the eigenspace $\langle {\mathbf 1}\rangle$ is fastest, provided 0 is a simple eigenvalue -- if all vertices are identified to form a flower graph.

\section{Appendix}\label{sec:app}

We argue that in the context of parabolic problems associated with different realizations of the bi-Laplacian, \textit{eventual positivity} and its relaxation, \textit{asymptotic positivity} are more appropriate notions: in the context of operators on general Banach lattices they have been thoroughly discussed in~\cite{DanGluKen16b,DanGluKen16}. For the sake of self-containedness, let us present a reminder of the most relevant results obtained so far in this area.

Recall that if $X$ is a $\sigma$-finite measure space, then $L^2(X)$ is a Hilbert lattice~\cite{Nag86}; if we denote its positive cone by $L^2(X)_+$, then for all $u\in L^2(X)_+$ the \textit{principal ideal} generated by $u$ is the set
 \[
 L^2(X)_u:=\{f\in L^2(X):\exists c>0:|f|\le cu\}\ .
 \]
We write $f\gg_u 0$ if $f$ is real and there exists $\epsilon>0$ such that $f\ge \epsilon u$. Observe that if $f>0$ a.e., then $L^2(X)_u$ is dense in $L^2(X)$, i.e., $u$ is a \textit{quasi-interior} point of $L^2(X)_+$ in the language of Banach lattice theory; we can thus formulate the above-mentioned results as follows.

\begin{defi}
Let $X$ be a $\sigma$-finite measure space and $L^2(X)_+$   be the  positive cone of the Hilbert lattice  $L^2(X)$. A strongly continuous semigroup $(T(t))_{t\ge 0}$ is said to be 
\begin{itemize}
\item \emph{individually asymptotically positive} if 
\[
f\ge 0\Rightarrow \lim_{t\to\infty}\dist(e^{-ts(-A)}T(t)f,L^2(X)_+)=0,
\]
provided the spectral bound $s(-A)$ of its generator $-A$ is finite.
\item \emph{uniformly eventually strongly positive}  with respect to $u\in L^2(X)_+$ if
\[
\exists t_0>0\hbox{ s.t.}\quad f\ge 0,\ f\ne 0 \Rightarrow T(t)f\gg_u  0\quad \hbox{for all }t\ge t_0.
\]
\end{itemize}
\end{defi}

If $X$ has finite measure and thus $u$ can be taken to be the constant function ${\bf 1}$, then $f\gg_u 0$ is equivalent to $f\gg 0$ and so \textit{uniform eventual strong positivity with respect to ${\bf 1}$} is nothing but \textit{eventual irreducibility}, a rather strong property. At the other end of the spectrum, {individual asymptotic positivity} seems to be the weakest conceivable condition that still gives a scintilla of positivity to a semigroup.

Let us recall the following corollaries of~\cite[Thm.~10.2.1]{Glu16} and~\cite[Thm.~8.3]{DanGluKen16}.

\begin{prop}\label{prop:nonstrongly}
Let $X$ be a $\sigma$-finite measure space and $(T(t))_{t\ge 0}$ be a  strongly continuous semigroup 
on the complex Hilbert lattice $L^2(X)$ with self-adjoint generator $-A$. Assume that $s(-A)>-\infty$, that $(e^{-t(A+s(-A))})_{t\ge 0}$ is bounded, and that $s(-A)$ is a pole of the resolvent. 

Then $(T(t))_{t\ge 0}$ is individually asymptotically positive if and only if $s(-A)$ is a dominant spectral value of A and the associated spectral projection $P$ is positive.
\end{prop}

\begin{prop}\label{prop:indivunif}
Let $X$ be a $\sigma$-finite measure space and $(T(t))_{t\ge 0}$ be a real, strongly continuous semigroup 
on the complex Hilbert lattice $L^2(X)$ with self-adjoint generator $-A$. Let $u>0$ a.e.\ and assume that $\bigcap_{k=1}^\infty D(A^k)\subset L^2(X)_u$.

 Then $(T(t))_{t\ge 0}$ is uniformly eventually strongly positive  with respect to $u$
 if and only if the spectral bound $s(-A)$ is a simple eigenvalue and the associated eigenspace contains a vector $v$ such that $v\gg_u 0$.
\end{prop}

It is remarkable that under the assumption of Proposition~\ref{prop:indivunif}, the semigroup is necessarily of trace class and hence its generator has purely point spectrum; whereas Proposition~\ref{prop:nonstrongly} may apply also in the case of a generator whose spectrum is not purely discrete, but which does have a dominant eigenvalue.

We can deduce the following result for semigroups associated with forms from the Beurling--Deny criteria.

\begin{cor}\label{cor:indivunif}
Let $X$ be a finite measure space and $a$ be a closed symmetric form on $L^2(X)$; denote by $A$ the associated operator. Then the following assertions hold.

\begin{enumerate}[(1)]
\item $(e^{-tA})_{t\ge 0}$ is individually asymptotically positive if and only if $s(-A)$ is a dominant spectral value of $-A$ and the associated spectral projection is positive.

\item Let additionally $D(a)\subset L^\infty(X)$ and $\Real u\in D(a)$ and $a(\Real u, \Ima u)\in \R$ for all $u\in D(a)$. 
Then $(e^{-tA})_{t\ge 0}$ is eventually irreducible if and only if the spectral bound $s(-A)$ is a simple eigenvalue and the associated eigenspace contains a vector $v$ such that $v\gg 0$.
\end{enumerate}
\end{cor}

In the same spirit one may ask whether a given semigroup is \textit{eventually} $L^\infty$-contractive. This issue was not investigated in~\cite{DanGluKen16,DanGluKen16b} but an answer can be given in a simple case that covers, in particular, finite graphs and networks.

Our result seems to be of independent interest. It has been observed by the second author together with Jochen Glück (Ulm).
 
\begin{prop}\label{prop:glumug}
Let $X$ be a finite measure space and $A$ a self-adjoint operator with compact resolvent on $L^2(X)$. Let $0$ be the spectral bound of $A$ and let the associated eigenspace $E_0$ be a one-dimensional space spanned by $\bf 1$. Also assume that $\bigcap_{k\in \mathbb N}D(A^k)\hookrightarrow L^\infty(X)$. Then the semigroup generated by $A$ is uniformly eventually sub-Markovian, i.e., uniformly eventually positive and uniformly eventually $L^\infty$-contractive.
\end{prop}

\begin{proof}
First of all, let us observe that the eigenprojector onto $E_0$, which is the rank-one operator
\[
P:f\mapsto \frac{1}{|X|}\int_X |f|\ dx
\]
is also a positive, bounded linear operator on $L^\infty(X)$ with norm 1. It follows from Proposition~\ref{prop:indivunif} that the semigroup is uniformly eventually strongly positive -- say it is positive for all $t>t_0$.


Let now $t>t_0$ and hence $e^{tA}$ be positive: in order to show that $e^{tA}$ is $L^\infty$-contractive and hence sub-Markovian, it suffices to check that $e^{t_0 A}{\bf 1}= {\bf 1}$, i.e., $A{\bf 1}=0$. This holds by assumption.
\end{proof}

\bibliographystyle{alpha}
\bibliography{literatur}

\begin{thebibliography}{KKMM16}

\bibitem[ABHN01]{AreBatHie01}
W.\ Arendt, C.J.K.\ Batty, M.\ Hieber, and F.\ Neubrander.
\newblock {\em Vector-{V}alued {L}aplace {T}ransforms and {C}auchy {P}roblems},
  volume~96 of {\em Monographs in Mathematics}.
\newblock Birkh{\"a}user, Basel, 2001.

\bibitem[AN70]{AndNis70}
T.~Ando and K.~Nishio.
\newblock Positive selfadjoint extensions of positive symmetric operators.
\newblock {\em Tokohu Math.\ J.}, 22:65--75, 1970.

\bibitem[Are04]{Are04}
W.\ Arendt.
\newblock Semigroups and evolution equations: Functional calculus, regularity
  and kernel estimates.
\newblock In C.M.\ Dafermos and E.\ Feireisl, editors, {\em Handbook of
  Differential Equations: Evolutionary Equations -- Vol.\ 1}. North Holland,
  Amsterdam, 2004.

\bibitem[Are06]{Are06}
W.\ Arendt.
\newblock Heat {K}ernels -- {M}anuscript of the $9^{\rm th}$ {I}nternet
  {S}eminar, 2006.
\newblock (freely available at
  \url{http://www.uni-ulm.de/fileadmin/website_uni_ulm/mawi.inst.020/arendt/downloads/internetseminar.pdf}).

\bibitem[AtE17]{AreEls17}
W.\ Arendt and T.~ter Elst.
\newblock {The Dirichlet-to-Neumann operator on $ C(\partial\Omega)$}.
\newblock arXiv:1707.05556, 2017.

\bibitem[BD59]{BeuDen59}
A.\ Beurling and J.~Deny.
\newblock {Dirichlet spaces}.
\newblock {\em Proc.\ Natl.\ Acad.\ Sci.\ USA}, 45:208--215, 1959.

\bibitem[Bel85]{Bel85}
{J.\ von} Below.
\newblock A characteristic equation associated with an eigenvalue problem on
  $c^2$-networks.
\newblock {\em Lin.\ Algebra Appl.}, 71:309--325, 1985.

\bibitem[BK13]{BerKuc13}
G.\ Berkolaiko and P.~Kuchment.
\newblock {\em {Introduction to Quantum Graphs}}, volume 186 of {\em Math.\
  Surveys and Monographs}.
\newblock Amer.\ Math.\ Soc., Providence, RI, 2013.

\bibitem[BKKMar]{BerKenKur19}
G.~Berkolaiko, J.B. Kennedy, P.~Kurasov, and D.~Mugnolo.
\newblock Surgery principles for the spectral analysis of quantum graphs.
\newblock {\em Trans.\ Amer.\ Math.\ Soc.}, (to appear).

\bibitem[BL04]{BorLaz04}
A.V. Borovskikh and K.P. Lazarev.
\newblock Fourth-order differential equations on geometric graphs.
\newblock {\em J.\ Math.\ Sci.}, 119:719--738, 2004.

\bibitem[BM15]{BonMaz15}
S.~Bonaccorsi and S.~Mazzucchi.
\newblock High order heat-type equations and random walks on the complex plane.
\newblock {\em Stochastic Process.\ Appl.}, 125:797--818, 2015.

\bibitem[CDKP87]{CheDelKra87}
G.~Chen, M.C. Delfour, A.M. Krall, and G.~Payre.
\newblock Modeling, stabilization and control of serially connected beams.
\newblock {\em SIAM J.\ Control Opt.}, 25:526--546, 1987.

\bibitem[CFL28]{CouFriLew28}
R.~Courant, K.~Friedrichs, and H.~Lewy.
\newblock {\"Uber die partiellen Differenzengleichungen der mathematischen
  Physik}.
\newblock {\em Math.\ Ann.}, 100:32--74, 1928.

\bibitem[CM09]{CarMug09}
S.\ Cardanobile and D.\ Mugnolo.
\newblock Parabolic systems with coupled boundary conditions.
\newblock {\em J.\ Differ.\ Equ.}, 247:1229--1248, 2009.

\bibitem[Dav95a]{Dav95b}
E.B. Davies.
\newblock Long time asymptotics of fourth order parabolic equations.
\newblock {\em Journal d'Analyse Math{\'e}matique}, 67:323--345, 1995.

\bibitem[Dav95b]{Dav95}
E.B. Davies.
\newblock Uniformly elliptic operators with measurable coefficients.
\newblock {\em J.\ Funct.\ Anal.}, 132:141--169, 1995.

\bibitem[Dav07]{Dav07}
E.B. Davies.
\newblock {\em {Linear Operators And Their Spectra}}.
\newblock Cambridge Univ.\ Press, Cambridge, 2007.

\bibitem[DGK16a]{DanGluKen16b}
D.~Daners, J.~Gl{\"u}ck, and J.B. Kennedy.
\newblock Eventually and asymptotically positive semigroups on {B}anach
  lattices.
\newblock {\em J.\ Differ.\ Equ.}, 261:2607--2649, 2016.

\bibitem[DGK16b]{DanGluKen16}
D.~Daners, J.~Gl{\"u}ck, and J.B. Kennedy.
\newblock Eventually positive semigroups of linear operators.
\newblock {\em J.\ Math.\ Anal.\ Appl.}, 433:1561--1593, 2016.

\bibitem[Die05]{Die05}
R.\ Diestel.
\newblock {\em Graph Theory}, volume 173 of {\em Graduate Texts in
  Mathematics}.
\newblock Springer-Verlag, Berlin, 2005.

\bibitem[DL88]{DauLio88}
R.\ Dautray and J.-L.\ Lions.
\newblock {\em {Mathematical Analysis and Numerical Methods for Science and
  Technology, Vol.\ 2}}.
\newblock Springer-Verlag, Berlin, 1988.

\bibitem[DN99]{DekNic99}
B.~Dekoninck and S.~Nicaise.
\newblock {Control of networks of Euler-Bernoulli beams}.
\newblock {\em ESAIM: Control, Optimisation and Calculus of Variations},
  4:57--81, 1999.

\bibitem[DN00]{DekNic00}
B.~Dekoninck and S.~Nicaise.
\newblock The eigenvalue problem for networks of beams.
\newblock {\em Lin.\ Algebra Appl.}, 314:165--189, 2000.

\bibitem[DZ06]{DagZua06}
R.~D{\'a}ger and E.~Zuazua.
\newblock {\em Wave propagation, observation and control in 1-d flexible
  multi-structures}, volume~50 of {\em Math\'em.\ \& Appl.}
\newblock Springer-Verlag, Berlin, 2006.

\bibitem[Eng92]{Eng92}
K.-J. Engel.
\newblock On singular perturbations of second order {C}auchy problems.
\newblock {\em Pacific J.\ Math.}, 152:79--91, 1992.

\bibitem[FGG08]{FerGazGru08}
A.~Ferrero, F.~Gazzola, and H.-C. Grunau.
\newblock Decay and eventual local positivity for biharmonic parabolic
  equations.
\newblock {\em Disc.\ Cont.\ Dyn.\ Syst.}, 21:1129--1157, 2008.

\bibitem[fh]{Fed17}
fedja (https://mathoverflow.net/users/1131/fedja).
\newblock An elementary inequality for graph {L}aplacians.
\newblock MathOverflow.
\newblock \url{https://mathoverflow.net/q/287765 (version: 2017-12-05)}.

\bibitem[Fie73]{Fie73}
M.~Fiedler.
\newblock Algebraic connectivity of graphs.
\newblock {\em Czech.\ Math.\ J.}, 23:298--305, 1973.

\bibitem[FOT10]{FukOshTak10}
M.~Fukushima, Y.~Oshima, and M.~Takeda.
\newblock {\em {Dirichlet forms and symmetric Markov processes}}, volume~19 of
  {\em Studies in Math.}
\newblock de Gruyter, Berlin, 2010.

\bibitem[Fun79]{Fun79}
T.~Funaki.
\newblock Probabilistic construction of the solution of some higher order
  parabolic differential equation.
\newblock {\em Proc.\ Japan Acad., Ser.\ A}, 55:176--179, 1979.

\bibitem[GG08]{GazGru08}
F.~Gazzola and H.-C. Grunau.
\newblock {Eventual local positivity for a biharmonic heat equation in $\mathbb
  R^n$}.
\newblock {\em Disc.\ Cont.\ Dyn.\ Syst.\ S}, (83-87):265--266, 2008.

\bibitem[GH69]{GriHer69}
R.J. Griego and R.~Hersh.
\newblock Random evolutions, {M}arkov chains, and systems of partial
  differential equations.
\newblock {\em Proc.\ Natl.\ Acad.\ Sci.\ USA}, 62:305--308, 1969.

\bibitem[GH96]{GnuHaa96}
S.~Gnutzmann and F.~Haake.
\newblock Positivity violation and initial slips in open systems.
\newblock {\em Z.\ Phys.\ B}, 101:263--273, 1996.

\bibitem[GKO08]{GiaKnuOtt08}
L.~Giacomelli, H.~Kn{\"u}pfer, and F.~Otto.
\newblock Smooth zero-contact-angle solutions to a thin-film equation around
  the steady state.
\newblock {\em J.\ Differ.\ Equ.}, 245:1454--1506, 2008.

\bibitem[Gl{\"u}16]{Glu16}
J.\ Gl{\"u}ck.
\newblock {\em {Invariant Sets and Long Time Behaviour of Operator
  Semigroups}}.
\newblock PhD thesis, Universit{\"a}t Ulm, 2016.

\bibitem[GM94]{GM94}
J.~Gasch and L.~Maligranda.
\newblock On vector-valued inequalities of the marcinkiewicz-zygmund, herz and
  krivine type.
\newblock {\em Math. Nachr.}, 167:95--129, 1994.

\bibitem[GM16]{GreMil16}
F.\ Gregorio and S.~Mildner.
\newblock Fourth-order {S}chrödinger type operator with singular potentials.
\newblock {\em Arch.\ Math}, 107:285--294, 2016.

\bibitem[GMP14]{GalMitPoz14}
V.A. Galaktionov, E.L. Mitidieri, and S.I. Pohozaev.
\newblock {\em {Blow-up for Higher-Order Parabolic, Hyperbolic, Dispersion and
  Schrodinger Equations}}.
\newblock CRC Press, Boca Raton, FL, 2014.

\bibitem[Gra68]{Gra68}
B.\ Gramsch.
\newblock {Zum Einbettungssatz von Rellich bei Sobolevr\"aumen}.
\newblock {\em Math.\ Z.}, 106:81--87, 1968.

\bibitem[HBW99]{HanBenWei99}
S.M. Han, H.~Benaroya, and T.~Wei.
\newblock Dynamics of transversely vibrating beams using four engineering
  theories.
\newblock {\em J.\ Sound Vib.}, 225:935--988, 1999.

\bibitem[HK]{HK05}
O.~Holtz and M.~Karow.
\newblock Real and complex operator norms.
\newblock arxiv:math/0512608v1.

\bibitem[HLG18]{HufLinGal}
P.G. Hufton, Y.T. Lin, and T.~Galla.
\newblock Model reduction methods for classical stochastic systems with
  fast-switching environments: reduced master equations, stochastic
  differential equations, and applications.
\newblock arXiv:1803.02941, 2018.

\bibitem[Hoc78]{Hoc78}
K.J. Hochberg.
\newblock A signed measure on path space related to {W}iener measure.
\newblock {\em Ann.\ Probab.}, 6:433--458, 1978.

\bibitem[H{\"o}l39]{Hol39}
E.~H{\"o}lder.
\newblock Entwicklungss{\"a}tze aus der {T}heorie der zweiten {V}ariation --
  {A}llgemeine {R}andbedingungen.
\newblock {\em Acta Math.}, 70:193--242, 1939.

\bibitem[Kat54]{Kat54}
T.~Kato.
\newblock {On the semi-groups generated by Kolmogoroff's differential
  equations}.
\newblock {\em J.\ Math.\ Soc.\ Jap.}, 6:1--15, 1954.

\bibitem[KKMM16]{KenKurMal16}
J.B. Kennedy, P.~Kurasov, G.~Malenová, and D.~Mugnolo.
\newblock On the spectral gap of a quantum graph.
\newblock {\em Ann.\ Henri Poincar\'e}, 17:2439--2473, 2016.

\bibitem[KKU15]{KiiKurUsm15}
J.-C. Kiik, P.~Kurasov, and M.~Usman.
\newblock On vertex conditions for elastic systems.
\newblock {\em Phys.\ Lett.\ A}, 379:1871--1876, 2015.

\bibitem[KO02]{KohOtt02}
R.V.\ Kohn and F.~Otto.
\newblock Upper bounds on coarsening rates.
\newblock {\em Commun.\ Math.\ Phys.}, 229:375--395, 2002.

\bibitem[Kre47]{Kre47}
M.G. Krein.
\newblock The theory of self-adjoint extensions of semi-bounded hermitian
  transformations and its applications. {I}.
\newblock {\em Mat.\ Sbornik}, 20:431--495, 1947.

\bibitem[Kry60]{Kry60}
V.J. Krylov.
\newblock Some properties of the distribution corresponding to the equation
  $\partial u/\partial t=(-1)^{q+1}\partial^{2q}u/\partial x^{2q}$.
\newblock 1:760--763, 1960.

\bibitem[KS97]{KotSmi97}
T.\ Kottos and U.\ Smilansky.
\newblock Quantum chaos on graphs.
\newblock {\em Phys.\ Rev.\ Lett.}, 79:4794--4797, 1997.

\bibitem[KS99]{KosSch99}
V.\ Kostrykin and R.\ {S}chrader.
\newblock Kirchhoff's rule for quantum wires.
\newblock {\em J.\ Phys.\ A}, 32:595--630, 1999.

\bibitem[Kuc04]{Kuc04}
P.\ Kuchment.
\newblock Quantum graphs {I}: {S}ome basic structures.
\newblock {\em Waves Random Media}, 14:107--128, 2004.

\bibitem[LLS92]{LagLeuSch92}
J.E.\ Lagnese, G.\ Leugering, and E.J.P.G.\ Schmidt.
\newblock Modelling and controllability of networks of thin beams.
\newblock In {\em System Modelling and Optimization}, pages 467--480. Springer,
  1992.

\bibitem[LLS93]{LagLeuSch93}
J.E.\ Lagnese, G.\ Leugering, and E.J.P.G.\ Schmidt.
\newblock Modelling of dynamic networks of thin thermoelastic beams.
\newblock {\em Math.\ Meth.\ Appl.\ Sci.}, 16:327--358, 1993.

\bibitem[LLS94]{LagLeuSch94}
J.E.\ Lagnese, G.\ Leugering, and E.J.P.G.\ Schmidt.
\newblock {\em Modeling, {A}nalysis, and {C}ontrol of {D}ynamic {E}lastic
  {M}ulti-{L}ink {S}tructures}.
\newblock Systems and Control: Foundations and Applications. Birkh{\"a}user,
  Basel, 1994.

\bibitem[LP00]{LauPug00}
R.S. Laugesen and M.C. Pugh.
\newblock Linear stability of steady states for thin film and {Cahn-Hilliard}
  type equations.
\newblock {\em Arch.\ Ration.\ Mech.\ Anal.}, 154:3--51, 2000.

\bibitem[Lum80]{Lum80}
G.\ Lumer.
\newblock Connecting of local operators and evolution equations on networks.
\newblock In F.\ Hirsch, editor, {\em Potential Theory (Proc.\ Copenhagen
  1979)}, pages 230--243, Berlin, 1980. Springer-Verlag.

\bibitem[MN11]{MugNit11}
D.\ Mugnolo and R.\ Nittka.
\newblock Properties of representations of operators acting between spaces of
  vector-valued functions.
\newblock {\em Positivity}, 15:135--154, 2011.

\bibitem[Moh82]{Moh82}
B.~Mohar.
\newblock The spectrum of an infinite graph.
\newblock {\em Lin.\ Algebra Appl.}, 48:245--256, 1982.

\bibitem[MSW05]{MueShoWon05}
A.H. Mueller, A.I. Shoshi, and S.M.H. Wong.
\newblock Extension of the {JIMWLK} equation in the low gluon density region.
\newblock {\em Nuclear Phys.\ B}, 715:440--460, 2005.

\bibitem[Mug07]{Mug07}
D.\ Mugnolo.
\newblock Gaussian estimates for a heat equation on a network.
\newblock {\em Networks Het.\ Media}, 2:55--79, 2007.

\bibitem[Mug13]{Mug13}
D.~Mugnolo.
\newblock Parabolic theory of the discrete $p$-{L}aplace operator.
\newblock {\em Nonlinear Anal., Theory Methods Appl.}, 87:33--60, 2013.

\bibitem[Mug14]{Mug14}
D.~Mugnolo.
\newblock {\em {Semigroup Methods for Evolution Equations on Networks}}.
\newblock Underst.\ Compl.\ Syst. Springer-Verlag, Berlin, 2014.

\bibitem[Mug15]{Mug15b}
D.~Mugnolo.
\newblock Some remarks on the {K}rein-von {N}eumann extension of different
  {L}aplacians.
\newblock In J.~Banasiak, A.~Bobrowski, and M.~Lachowicz, editors, {\em
  Semigroups of Operators-Theory and Applications}, Proc.\ Math.\ \& Stat., New
  York, 2015. Springer-Verlag.

\bibitem[MVV05]{ManVogVoi05}
A.\ Manavi, H.\ Vogt, and J.\ Voigt.
\newblock {Domination of semigroups associated with sectorial forms}.
\newblock {\em J.\ Oper.\ Theory}, 54:9--25, 2005.

\bibitem[Nag86]{Nag86}
R.\ Nagel, editor.
\newblock {\em One-{P}arameter {S}emigroups of {P}ositive {O}perators}, volume
  1184 of {\em Lect.\ Notes Math.}
\newblock Springer-Verlag, Berlin, 1986.

\bibitem[Nic86]{Nic86}
S.\ Nicaise.
\newblock Probl{\'e}mes de {C}auchy pos{\'e}s en norme uniforme sur les espaces
  ramifi{\'e}s {\'e}l{\'e}mentaires.
\newblock {\em C.R.\ Acad.\ Sc.\ Paris S\'er.\ A}, 303:443--446, 1986.

\bibitem[Nic87]{Nic87}
S.\ Nicaise.
\newblock Spectre des r{\'e}seaux topologiques finis.
\newblock {\em Bull.\ Sci.\ Math., II.\ S{\'e}r.}, 111:401--413, 1987.

\bibitem[Nit12]{Nit12}
R.~Nittka.
\newblock Projections onto convex sets and {$L^p$}-quasi-contractivity of
  semigroups.
\newblock {\em Arch.\ Math.}, 98:341--353, 2012.

\bibitem[Ouh05]{Ouh05}
E.M. Ouhabaz.
\newblock {\em Analysis of {H}eat {E}quations on {D}omains}, volume~30 of {\em
  Lond.\ Math.\ Soc.\ Monograph Series}.
\newblock Princeton Univ.\ Press, Princeton, NJ, 2005.

\bibitem[Rel54]{Rel54}
F.~Rellich.
\newblock {Halbbeschr{\"a}nkte Differentialoperatoren h{\"o}herer Ordnung}.
\newblock In {\em Proc.\ Int.\ Congress of Mathematicians (Amsterdam 1954)},
  volume~3, pages 243--250, Amsterdam, 1954. North-Holland.

\bibitem[Rot83]{Rot83}
J.-P.\ Roth.
\newblock Spectre du laplacien sur un graphe.
\newblock {\em C.\ R.\ Acad.\ Sci.\ Paris S{\'e}r.\ I Math.}, 296:793--795,
  1983.

\bibitem[Sch12]{Sch12}
K.~Schmüdgen.
\newblock {\em {Unbounded Self-adjoint Operators on Hilbert Space}}, volume 265
  of {\em Graduate Texts in Mathematics}.
\newblock Springer-Verlag, Berlin, 2012.

\bibitem[SSO92]{SuaSilOpp92}
A.~Su{\'a}rez, R.~Silbey, and I.~Oppenheim.
\newblock Memory effects in the relaxation of quantum open systems.
\newblock {\em J.\ Chem.\ Phys.}, 97:5101--5107, 1992.

\end{thebibliography}
\end{document}